\documentclass[12pt,makeidx]{amsart}

\usepackage{amssymb,mathrsfs,amsmath, amsthm}
\usepackage{url,titletoc}
\usepackage{bbm,xspace}
\usepackage{mathrsfs,mathabx}
 
\usepackage{mathtools}
\usepackage{relsize}

\let\oldFootnote\footnote
\newcommand\nextToken\relax

\renewcommand\footnote[1]{%
    \oldFootnote{#1}\futurelet\nextToken\isFootnote}

\newcommand\isFootnote{%
    \ifx\footnote\nextToken\textsuperscript{,}\fi}
\allowdisplaybreaks

\usepackage[usenames,dvipsnames]{xcolor}
\definecolor{darkblue}{rgb}{0.0, 0.0, 0.55}

\usepackage[pagebackref,colorlinks,linkcolor=BrickRed,citecolor=darkblue,urlcolor=black,hypertexnames=true]{hyperref}

\usepackage{enumitem}

\usepackage{ulem}

\newtheorem{theorem}{Theorem}[section]
\newtheorem{cor}[theorem]{Corollary}
\newtheorem{lemma}[theorem]{Lemma}

\newtheorem{prop}[theorem]{Proposition}

\newtheorem{thm}[theorem]{Theorem}
\newtheorem{lem}[theorem]{Lemma}

\theoremstyle{definition}

\theoremstyle{remark}

\newtheorem{rem}[theorem]{Remark}

\def\bel{\begin{lemma}}
\def\eel{\end{lemma}}

\def\ben{\begin{enumerate}}
\def\een{\end{enumerate}}

\def\bem{\begin{pmatrix}}
\def\eem{\end{pmatrix}}

\def\bbm{\begin{bmatrix}}
\def\ebm{\end{bmatrix}}

\def\beq{\begin{equation}}
\def\eeq{\end{equation}}

\def\bep{\begin{proof}}
\def\eep{\end{proof}}

\DeclareMathOperator{\ran}{ran}

\def\ssec{\subsection}

\DeclareMathOperator{\real}{Re}
\DeclareMathOperator{\trc}{tr}

\DeclareMathOperator{\EndH}{B}

\newcommand{\irrL}{irreducible\xspace}
\newcommand{\irra}{atomic\xspace}
\newcommand{\irr}{atom\xspace}
\newcommand{\irrs}{atoms\xspace}

\newcommand{\C}{\mathbb{C}}
\newcommand{\R}{\mathbb{R}}

\def\tc{\tilde{c}}

\def\tv{\tilde{v}}

\def\tA{\widetilde{A}}
\def\tB{\widetilde{B}}
\def\tC{\widetilde{C}}

\def\tN{\tilde{N}}

\def\tL{\widetilde{L}}
\def\hL{\widehat{L}}
\def\vL{\widecheck{L}}

\def\cA{\mathcal{A}}

\def\cC{\mathcal{C}}
\def\cD{\mathcal{D}}

\def\cK{\mathcal{K}}

\def\cO{\mathcal{O}}

\def\cS{\mathcal{S}}

\def\cU{\mathcal{U}}
\def\cV{\mathcal{V}}

\def\cZ{\mathcal{Z}}

\def\trace{\operatorname{tr}}

\def\interior{\operatorname{int}}
\def\rr{\mathbbm{r}}
\def\ss{\mathbbm{s}}

\def\Span{\operatorname{span}}

\def\N{{\mathbb N}}

\makeatletter
\newcommand{\oset}[3][0ex]{%
  \mathrel{\mathop{#3}\limits^{
    \vbox to#1{\kern-2\ex@
    \hbox{$\scriptstyle#2$}\vss}}}}
\makeatother

\DeclareMathOperator{\mydot}{\raisebox{1pt}{$\mathsmaller{\bigodot}$}}

\newcommand{\ulx}{x}

\newcommand{\Langle}{\mathop{<}\!}

\newcommand{\Rangle}{\!\mathop{>}}
\newcommand{\mx}{\!\Langle \ulx\Rangle}
\newcommand{\mxs}{\!\Langle \ulx^*\Rangle}
\newcommand{\mxx}{\!\Langle \ulx,\ulx^*\Rangle}
\newcommand{\px}{\C\!\Langle \ulx,\ulx^*\Rangle}
\newcommand{\pxa}{\C\!\Langle \ulx\Rangle}
\newcommand{\pxs}{\C\!\Langle \ulx^*\Rangle}
\newcommand{\pz}{\C\!\Langle z \Rangle}

\makeatletter
\def\moverlay{\mathpalette\mov@rlay}
\def\mov@rlay#1#2{\leavevmode\vtop{
    \baselineskip\z@skip \lineskiplimit-\maxdimen
    \ialign{\hfil$#1##$\hfil\cr#2\crcr}}}
\makeatother

\newcommand{\plangle}{\moverlay{(\cr<}}
\newcommand{\prangle}{\moverlay{)\cr>}}
\newcommand{\rx}{\C\plangle x \prangle}
\newcommand{\rxx}{\C\plangle x,x^* \prangle}
\newcommand{\rz}{\C\plangle z \prangle}

\newcommand*{\mat}[1]{\operatorname{M}_{#1}(\C)}
\newcommand*{\matt}[2]{\operatorname{M}_{{#1}}(#2)}

\newcommand*{\zar}[1]{\overline{#1}^{\rm\,Zar}}

\def\de{\delta}
\def\ve{\varepsilon}
\def\cVh{\cV^{\rm h}}
\def\pL{L^{\prime}}
\def\bb{{\bf b}}

\def\tLx{\widetilde{L}_{\times}}
\def\hc{\widehat{c}}
\numberwithin{equation}{section}

\newcommand{\df}[1]{{\bf{#1}}{\index{#1}}}

\makeindex
\makeatletter
\newcommand{\mycontentsbox}{%
\printindex
{\centerline{NOT FOR PUBLICATION}
\addtolength{\parskip}{-2.0pt}\footnotesize
\tableofcontents}}
\def\enddoc@text{\ifx\@empty\@translators \else\@settranslators\fi
\ifx\@empty\addresses \else\@setaddresses\fi
\newpage\mycontentsbox%\newpage\printindex
}
\makeatother

\contentsmargin{2.55em} 

\title{Noncommutative polynomials describing convex sets}
\author[J.W. Helton]{J. William Helton${}^1$}
\address{J. William Helton, Department of Mathematics\\
  University of California \\
  San Diego}
\email{helton@math.ucsd.edu}
\thanks{${}^1$Research supported by the NSF grant
DMS-1500835.}

\author[I. Klep]{Igor Klep${}^{2}$}
\address{Igor Klep, Faculty of Mathematics and Physics, 
University of Ljubljana, Slovenia}
\email{igor.klep@fmf.uni-lj.si}
\thanks{${}^2$Supported by the Slovenian Research Agency grants J1-8132, N1-0057 and P1-0222. Partially supported by the 
Marsden Fund Council of the Royal Society of New Zealand.}

\author[S. McCullough]{Scott McCullough${}^3$}
\address{Scott McCullough, Department of Mathematics,  University of Florida, Gainesville
   }
   \email{sam@math.ufl.edu}
\thanks{${}^3$Research supported by  NSF grants DMS-1361501 and DMS-1764231.}

\author[J. Vol\v{c}i\v{c}]{Jurij Vol\v{c}i\v{c}${}^4$}
\address{Jurij Vol\v{c}i\v{c}, Department of Mathematics\\
	Texas A\&M University \\ College Station}
\email{volcic@math.tamu.edu}
\thanks{${}^4$Research supported by the Deutsche Forschungsgemeinschaft (DFG) Grant No. SCHW 1723/1-1.}

\date{\today}  


\begin{document}

%\today

\dottedcontents{section}[3.8em]{}{2.3em}{.4pc} 
\dottedcontents{subsection}[6.1em]{}{3.2em}{.4pc}
\dottedcontents{subsubsection}[8.4em]{}{4.1em}{.4pc}

%\tableofcontents

\numberwithin{equation}{section}

\begin{abstract}
 The free closed semialgebraic set $\cD_f$ determined by a hermitian noncommutative
 polynomial $f\in \matt\de\px$ is 
the closure of the  connected component of $\{(X,X^*)\mid f(X,X^*)\succ0\}$ 
     containing the origin. When $L$ is a hermitian monic linear pencil,
 the free closed semialgebraic set $\cD_L$ is 
 the feasible set of the linear matrix inequality $L(X,X^*)\succeq 0$ and is known as
 a free spectrahedron. Evidently these are convex and 
 it is well-known that a free closed semialgebraic set 
   is convex if and only it is a free spectrahedron.
 The main result of this paper solves the basic problem of determining
 those $f$ for which $\cD_f$ is convex. The solution leads to
 an efficient algorithm that not only determines if $\cD_f$ is convex, but if so,
 produces a minimal hermitian monic pencil $L$ such that $\cD_f=\cD_L$.
 Of independent interest is a subalgorithm based on a Nichtsingul\"arstellensatz presented here:  given
  a linear pencil $\tL$ and a hermitian monic pencil $L$, it determines if 
$\tL$ takes invertible values on the interior of $\cD_L.$ 
  Finally, 
    it is shown that if $\cD_f$ is convex for
an irreducible hermitian
     $f\in\px$,
  then  $f$ has degree at most two, and arises as the
	Schur complement of an  $L$ such that $\cD_f=\cD_L$.
\end{abstract} 

\maketitle

%---------------  ARXIV abstract
\iffalse
The free closed semialgebraic set $D_f$ determined by a hermitian noncommutative polynomial $f$ is the closure of the connected component of $\{(X,X^*)\mid f(X,X^*)>0\}$ containing the origin. When $L$ is a hermitian monic linear pencil, the free closed semialgebraic set $D_L$ is the feasible set of the linear matrix inequality $L(X,X^*)\geq 0$ and is known as a free spectrahedron. Evidently these are convex and it is well-known that a free closed semialgebraic set is convex if and only it is a free spectrahedron. The main result of this paper solves the basic problem of determining those $f$ for which $D_f$ is convex. The solution leads to an efficient algorithm that not only determines if $D_f$ is convex, but if so, produces a minimal hermitian monic pencil $L$ such that $D_f=D_L$. Of independent interest is a subalgorithm based on a Nichtsingulärstellensatz presented here: given a linear pencil $L'$ and a hermitian monic pencil $L$, it determines if $L'$ takes invertible values on the interior of $D_L$. Finally, it is shown that if $D_f$ is convex for an irreducible hermitian polynomial $f$, then $f$ has degree at most two, and arises as the Schur complement of an $L$ such that $D_f=D_L$. 

%---------------------
\fi

\section{Introduction}

Semidefinite programming (SDP) \cite{Nem,WSV} is the main branch
of convex optimization to emerge in the last 25 years. 
Feasibility sets of semidefinite programs are given by linear matrix inequalities (LMIs)
and are called spectrahedra. 
We refer to the book \cite{BPR13} for an overview of the substantial theory of LMIs and spectrahedra and the connection to real algebraic geometry.
Spectrahedra are now basic objects in a number of  areas of mathematics. They figure prominently in determinantal representations \cite{Bra11, GK-VVW, NT12, PV, Vin93}, in the solution of the Kadison-Singer paving conjecture \cite{MSS15}, and the solution of the Lax conjecture \cite{HV07,LPR}.

One of the main applications of SDP lies in linear systems and control theory \cite{SIG97}.   From both empirical observation and the textbook classics, one sees that many problems in
this subject are described by signal flow diagrams and
naturally convert to inequalities involving polynomials in matrices. These polynomials
depend only upon the signal flow diagram and are otherwise independent of either the matrices
or their sizes. % They are \df{dimension-free}.
 Thus many problems in  systems and control  naturally lead to noncommutative polynomials, or more generally rational functions and matrix inequality conditions. This paper solves the basic problem of identifying those noncommutative
 rational matrix inequalities that give rise to convex feasibility sets. 
For expository purposes, the body of the paper presents a detailed proof of this fact for 
 noncommutative polynomials. The modifications needed to handle the
 more technically challenging case of matrix rational functions are 
 indicated in Appendix~\ref{sec:modify}.

 %We point out that in the dimension-free setting convexity is governed by LMIs \cite{EW97,HM12}.} 

The main results of the article are stated in this introduction. Following a review of basic definitions including that of a free spectrahedron and free semialgebraic set in Subsection \ref{sec:defs}, the three main results are presented in Subsection \ref{ssec:main}
followed by a guide to the paper in Subsection \ref{ssec:guide}.

\subsection{Definitions}
\label{sec:defs}
Let $x=(x_1,\dots,x_g)$ denote  freely noncommuting variables
and $x^*=(x_1^*,\dots,x_g^*)$ their formal adjoints.
Let $\mxx$\index{$\mxx$} denote the set of words in $x$ and $x^*$
and  $\px$\index{$\px$} the free polynomials in $(x,x^*)$
 equal  the finite $\C$-linear combinations from $\mxx$. 
For a positive integer $\de$, the set of free polynomials
with coefficients in $\mat{\de}$ 
is denoted  $\matt{\de}{\px}$\index{$\matt\de\px$} and is naturally identified
with the tensor product  $\mat{\de}\otimes \px$. 
  The ring $\px$ has a natural involution ${}^*$ 
 determined by sending the variables
$x_j$ to $x_j^*$, and vice-versa, sending scalars to their
complex conjugates and $(fg)^* = g^* f^*$ for $f,g\in \px$.
An element $f\in \px$ is hermitian if $f=f^*.$ This involution, and the notion of a
 hermitian polynomial,  naturally extends to $\matt{\de}{\px}$.

An element $f\in \matt{\de}{\px}$
is a finite sum
\beq\label{eq:mtxPoly}
f=\sum_{w\in\mxx}f_w w\in\mat{\de} \otimes \px
 = \matt{\de}{\px},
\eeq
where $f_w\in \mat{\de}$.
Given a  $g$-tuple  $X=(X_1,\dots,X_g)\in \mat{n}^g$,  a word $w\in \mxx$ is 
evaluated at $(X,X^*)$ in the natural way, resulting in an $n\times n$ matrix  $w(X,X^*)$.
The polynomial $f$ of equation \eqref{eq:mtxPoly}
is then evaluated at $X$ as
\[
f(X,X^*)=\sum_{w\in\mxx} f_w \otimes w(X,X^*) \in \mat{\de}\otimes \mat{n} = \mat{n\de}.
\]
It is a standard fact
that $f$ is hermitian if and only if $f^*(X,X^*) =f(X,X^*)^*$ for each $n$ 
and $X\in \mat{n}^g$.

Affine linear polynomials play a special role.
  A \df{monic (linear) pencil} of size $\de$ is an
  element $L$ of $\matt{\de}{\px}$  of the form
\begin{equation}
\label{eq:pencil}
 L(x,x^*) 
 = I_\de - A\mydot x - B\mydot x^*=
   I_\de - \sum_{j=1}^g  A_j x_j  - \sum_{j=1}^g B_j x_j^*.
\end{equation}
In the  case $B=A^*,$ the pencil $L$ is  a \df{hermitian  monic (linear) pencil}. 
The associated spectrahedron
\[
\cD_L(n) =\{(X,X^*)\in \mat{n}^{2g}: L(X,X^*)\succeq 0\}\footnote{For a square matrix $T$, the notation $T\succeq 0$ indicates that $T$ is positive semidefinite.}
\]
is a convex semialgebraic set and is the closure of the connected 
set $\{(X,X^*)\in\mat{n}^{2g} : L(X,X^*)\succ 0\}$.
The union, over $n$, of  $\cD_L(n)$ is a \df{free spectrahedron}, denoted $\cD_L$.
\index{$\cD_L$}

Given $f\in \matt{\de}{\px}$ with $\det f(0)\neq0$ and a positive integer $n$, let
$\cK_f(n)$ denote the closure of the connected component of $0$ of 
\[
 \{(X,X^*)\in \mat{n}^{2g}: \det f(X,X^*)\neq0\}.
\]
The \df{free invertibility set} $\cK_f$ associated to $f$ is then the union,
over $n$, of the $\cK_f(n)$. By replacing $f$ by $ f(0)^{-1} f$ we may, and usually do, assume that $f(0)=I$.
 A free invertibility set $\cK_f$ is \df{convex} if  \index{$\cK_f$}
each $\cK_f(n)$ is. 
If $f=f^*$ is hermitian, then
$\cK_f$ is a \df{free semialgebraic set} that we denote here by  $\cD_f$. \index{$\cD_f$}
(Letting $g=f^*f$, we see that
$g$ is hermitian with $g(0)=I$, 
 and $\cK_f = \cK_g =\cD_g$.)
In particular, if $L$ is a hermitian monic pencil, then
 $\cD_L$ is a convex free semialgebraic set.
Questions surrounding convexity of free semialgebraic sets arise in applications
such as systems engineering and are natural from the point of view of 
the theories of completely positive maps, operator systems and matrix convex sets \cite{Pau,EW97}, and quantum information theory \cite{HKMjems,BN}.
It is known, \cite{HM12,Kri}, that $\cK_f$ is convex if and only if there is an
 hermitian monic pencil  $L$ such that $\cK_f=\cD_L$.

\subsection{Main results}\label{ssec:main}
 We are now ready to  exposit our main
 results. 
 Using the theory of realizations for
noncommutative rational functions \cite{BGM05,BR,GGRW05,KVV09,Vol17}, in Theorem \ref{thm:main2} we
  explicitly and constructively describe the  structure of noncommutative matrix polynomials $f$
whose invertibility set $\cK_f$
is convex. 
Each  $\de\times\de$ noncommutative polynomial
or noncommutative rational function 
$r$ with $r(0)=I$  has  a 
\df{noncommutative Fornasini-Marchesini (FM) realization}. Namely,
 there exists a positive integer $d$ (the size of the realization), a monic linear
pencil with coefficients from $\mat{d}$, and 
$c,b_1,\dots, b_{2g} \in \mat{d\times \de}$ such that
\begin{equation}
\label{e:fm}
r(x,x^*) = I_\de + c^* L(x,x^*)^{-1} \bb,
\end{equation}
where $\bb:=\sum_{j=1}^g  (b_j x_j + b_{g+j} x_j^*)$.  
A  $d\times d$ linear pencil $L$
as in \eqref{eq:pencil} is \df{\irrL} if $A_1,\dots,A_g,B_1,\ldots,B_g$ generate $\mat{d}$ as a $\C$-algebra.
For non-constant $r$,
the FM realization  \eqref{e:fm} is \df{minimal} if $L$ has minimal size
amongst all FM realizations of $r$.\footnote{It is convenient
 to declare the size of a minimal realization for constant $r$
 to be $0$.} Since any two minimal realizations are
equivalent up to change of basis (see also 
 Remark \ref{r:facts}  for details),
Theorem~\ref{thm:main2} below
 does not depend upon the choice of minimal realization.

\begin{thm}\label{thm:main2}
Let $f\in\matt{\de}{\px}$ with $f(0)=I$.
Let $f^{-1}=I+c^* L^{-1}\bb$ be
a minimal  FM realization.
After a basis change we can assume that
\begin{equation}\label{e:triang}
L=
\begin{pmatrix}
L^1 & \star & \star \\
& \ddots &\star \\
& & L^\ell
\end{pmatrix},
\end{equation}
with each $L^i$  either \irrL or an identity matrix.

Let $\hL$ be the direct sum of those \irrL blocks $L^i$ of $L$ that are similar to a hermitian monic pencil, and let $\vL$
be the direct sum of the remaining $L^j$.
Then the following are equivalent:
\begin{enumerate}[label={\rm(\roman*)}]
\item \label{it:convex} $\cK_f$ is convex; 
\item \label{it:spec}
$\cK_f$ is a free spectrahedron;
\item\label{it:midInv}
$\cK_f=\cK_{\hL}$;
\item\label{it:lastInv}
$\vL$ is invertible on the interior
of $\cK_{\hL}$.
\end{enumerate}
\end{thm}

\begin{proof}
 If $\cK_f$ is convex, then
it is a free spectrahedron (by \cite{HM12}). Hence \ref{it:convex} implies \ref{it:spec}.
The converse is immediate.  The equivalence of items \ref{it:midInv} and \ref{it:lastInv}
 is straightforward.  Evidently item \ref{it:midInv} implies \ref{it:spec}.
The converse is proved in  Section \ref{ssec:pf12}. 
\end{proof}

Theorem \ref{thm:main2} 
implies that, for a monic linear pencil $L$, the invertibility set
$\cK_L$ is convex
if and only if
the semisimple part of a minimal size pencil $L$ describing $\cK_L$ is 
similar to a hermitian pencil.

  A non-invertible element $f\in\matt{\de}{\px}$ with $\det f(0)\neq0$ is an \df{atom} (\cite[Section 3.2]{Coh06}) if it does not factor; that is, it can not  be written as $f_1f_2$ for non-invertible
 $f_j\in \matt{\de}{\px}$. 
Given $f_j\in\matt{\de_j}{\px}$ for $1\le j\le t$, the intersection $\cK:=\bigcap_j\cK_{f_j}$ is
 \df{irredundant} if $\cK_{f_j}\not\subseteq \bigcap_{k\neq j}\cK_{f_k}$ for all $j$.
Theorem \ref{thm:main2} yields the following striking result providing further evidence of 
the rigid nature of convexity for free semialgebraic sets.

\begin{cor}\label{cor:sacvx}
Suppose $f_j\in\matt{\de_j\times\de_j}{\px}$ 
are atoms with $f_j(0)=I$. If $\cK:=\bigcap_j \cK_{f_j} $
is irredundant,  then $\cK$ is convex if and only if each $\cK_{f_j}$ is convex.
\end{cor}

\begin{proof} See  Subsection \ref{ssec:pf12}. \end{proof}

Theorem \ref{thm:main2}
 leads to algorithms
based on semidefinite programming. 
Note that 
Part \ref{it:algo2} of Corollary \ref{cor:algo12} below
asserts the existence of 
an effective version of the main
result of \cite{HM12}.

\begin{cor}\label{cor:algo12}
Let $f\in\matt{\de}{\px}$ with $\det f(0)\neq0$ be given. There is an efficient deterministic algorithm based on linear algebra and
semidefinite programming (SDP) to:
\begin{enumerate}[label={\rm(\arabic*)}]
	\item\label{it:algo1}
	check whether
	$\cK_f$ is convex;
	\item\label{it:algo2}
    (in the case $\cK_f$ is convex)	compute
	a linear matrix inequality (LMI) representation for 
	$\cK_f$; that is, a hermitian monic pencil 
	$L$ (of minimal size) with $\cK_f=\cD_{L}$.
\end{enumerate}
\end{cor}
An  SDP can be solved up to a given
  arbitrary precision in polynomial time \cite[Section 6.4]{NN}.
Thus in practice our algorithm
runs in polynomial time.
The proof of \ref{it:algo2} is based on
Theorem \ref{thm:main2} (see Subsection \ref{ssec:algo}), while the proof
of \ref{it:algo1} in Subsection \ref{ssec:algo2} uses
\ref{it:algo2} and new, of independent interest,  (recursive) certificates
for invertibility of linear pencils on interiors of free spectrahedra.

\begin{thm}[Nichtsingul\"arstellensatz]\label{thm:rank}\index{Nichtsingul\"arstellensatz}
Let $L$ be a hermitian monic pencil,
and let  $\tL$ be a not necessarily square affine linear matrix polynomial.
Consider the set of all matrices $D,C_k,P_0$ 
such that
$P_0\succeq0$ and
\begin{equation}\label{e:45intro}
D\tL+\tL^*D^*=P_0 + \sum_k C_k^* L C_k.
\end{equation}
(Such certificates can be searched for using semidefinite programming.)
\begin{enumerate}[label={\rm(\arabic*)}]
\item
If the only solutions of \eqref{e:45intro} have $P_0=0=C_k$, 
 then for some
$(X,X^*)$ in the interior of $\cD_L$, the matrix $\tL(X,X^*)$ is rank deficient;
\item
Otherwise let
$V=\ker P_0 \cap \bigcap_k \ker C_k$. 
\begin{enumerate}[label={\rm(\alph*)}]
\item
If $V=\{0\}$, then $\tL$ is full rank on
$\interior \cD_L$.
\item
If $V\neq\{0\}$, then $\tL$ is full rank on
$\interior \cD_L$ 
if and only if $\tL|_V$ is full rank on 
$\interior \cD_L$ and the theorem now applies with $\tL$ replaced by the smaller
pencil  $\tL|_V$.\looseness=-1
\end{enumerate}
\end{enumerate}
\end{thm}

\begin{proof}
See Proposition \ref{p:no}, Corollary \ref{c:rank} and its proof.
\end{proof}
 
For the special case of hermitian atoms 
with $\de=1$
the conclusion of Theorem \ref{thm:main2} can be significantly strengthened as the 
  final main result shows.

\begin{theorem}
\label{thm:main}
Suppose $f\in\px$ is a hermitian atom  and  $f(0)\succ 0$. 
If $\cD_f$ is proper and convex, then $f$ is of degree at most two, is  concave and is 
the Schur complement of any minimal size hermitian monic  
pencil $L$ satisfying $\cD_f=\cD_L$. 
\end{theorem}

\begin{proof}
See Section \ref{sec:pfmain}.
\end{proof}

Theorem \ref{thm:main} settles \cite[Conjecture 1.4]{DHM07}.
	In \cite[Theorem 5.4]{Vol2}, this result is further extended to the case when $f$ is not necessarily an atom but $\cD_{f+\ve}$ is proper and convex for all small $\ve>0$.

Noncommutative, or more accurately, freely noncommutative
 analysis has implications in the commutative setting,
particularly for LMIs.
Given a hermitian monic pencil $L$ 
the  set $\cD_L(1)$,  \df{level $1$} of the free spectrahedron
$\cD_L$,  
consisting  of  $\xi\in\C^g$ such that
$L(\xi,\overline{\xi})\succeq 0$
is a \df{spectrahedron} \cite{Viz}. 
Spectrahedra are currently of intense interest  in a number of areas; e.g.,
real algebraic geometry \cite{BPR13,Lau14,Tho},  optimization 
\cite{Nem,WSV,FGPRT15} and quantum information theory \cite{LP15,PNA}.
Problems involving free spectrahedra are typically tractable semidefinite programming
problems. Thus elevating  a problem involving spectrahedra 
to its free analog often  produces a tractable relaxation. The matrix cube 
problem of \cite{BN02,Nem} is a notable example of this phenomena \cite{HKM,HKMS}. See also 
\cite{DDOSS,KTT}. Theorem \ref{thm:rank} provides another example as it gives
a computationally tractable relaxation for the problem of determining whether
a polynomial is of constant sign on the interior of a spectrahedron.

\subsection{Reader's guide}\label{ssec:guide}
Section \ref{sec:prelim} contains background and some preliminary results on linear pencils, free
spectrahedra and realizations of noncommutative rational functions needed in the sequel. 
The proof of Theorem \ref{thm:main} is
given in Section \ref{sec:pfmain}, followed
by the proof of Theorem \ref{thm:main2}
and its corollary, Corollary \ref{cor:sacvx},
in Subsection \ref{ssec:pf12}. Corollary \ref{cor:algo12} and Theorem  
 \ref{thm:rank} are proved in the remainder of Section \ref{sec:algo}.
Subsection \ref{ssec:algo} contains
an algorithm
that, for a given noncommutative polynomial $f$ with convex $\cK_f$, constructs
a hermitian monic pencil
   $\widehat{L}$ with $\cD_{\widehat{L}}=\cK_f$.
Indeed, up to similarity, $\widehat{L}$
 is extracted from the monic linear
pencil $L$ appearing in a minimal FM realization of $f^{-1}$.
Subsection \ref{ssec:algo2} presents an efficient
algorithm for checking whether $\cK_f$ is convex. It is based on (the proof of) Theorem \ref{thm:main2} and representation theory and produces a 
finite sequence of semidefinite programs of decreasing size whose feasibility determines if $\cK_f$ is convex.
Section \ref{s:exa} presents several illustrative examples establishing optimality of our  main results. Further, Subsection \ref{ss:e4}
settles a conjecture from \cite{DHM07} on the degrees of atoms $f$ with convex 
$\cK_f$ in the negative.
In Section \ref{s:flip} we characterize
hermitian monic pencils that can arise in a minimal
realization of a noncommutative polynomial;
these pencils underpin our constructions
in Section \ref{s:exa}. Finally,
Section \ref{s:hered} provides a detailed analysis
of factorizations of hereditary noncommutative polynomials.
As a consequence, an
 hereditary minimal degree defining polynomial for a free spectrahedron is an atom, 
and hence has degree at most two, see Corollary \ref{cor:heredDefpoly}.

\section{Preliminaries}\label{sec:prelim}
Let $z=(z_1,\dots,z_g,z_{g+1},\dots,z_{2g})=(x_1,\dots,x_g,y_1,\dots,y_g)$
denote $2g$ freely noncommuting variables. Replacing $z_{g+j}=y_j$ with
$x_j^*$ identifies $\pz$ with $\px$. 
On the other hand, elements 
$f\in\pz$ are naturally evaluated at tuples 
$Z=(X,Y)\in \mat{n}^g\times \mat{n}^g = \mat{n}^{2g}$;  whereas
we evaluate  $f\in \px$ at $(X,X^*) \in \mat{n}^{2g}.$
The
use of $\pz$ versus $\px$ only signals our intent on 
viewing the domain of $f$ as either $\mat{n}^{2g}$ or
$\{(X,X^*):X\in \mat{n}^g\}\subset \mat{n}^{2g}$ respectively.
Indeed, we can identify $\pz$ with $\px$ 
whenever we work with  attributes of free polynomials
 that are {\it per se} independent of evaluations. 
For example, ring-theoretically there is no difference in using symbols $z_{g+j}$ 
instead of $x_j^*$ when talking about atomicity of polynomials.
 Therefore the results and definitions for matrix polynomials in $z=(z_1,\dots,z_{h})$, 
whose assumptions refer only to the structure, and not to evaluations, of polynomials, 
directly apply to matrix polynomials in $x_1,\dots,x_g,x_1^*,\dots,x_g^*$.

The \df{free locus} $\cZ_f$ of \index{$\cZ_f$}
 $f\in\pz^{\de\times\de}$  is the union, over $n\in\N$, of
%the hypersurfaces\looseness=-1
$$\cZ_f(n) =\left\{(X,Y)\in\mat{n}^{2g}\colon \det f(X,Y)=0\right\}.$$
Assuming  $\det f(0)\neq0$, as in the introduction, let $\cK_f=\bigcup_n\cK_f(n)$, where $\cK_f(n)$ is the closure of the connected component of
$$\left\{(X,X^*)\in\mat{n}^{2g}\colon \det f(X,X^*)\neq 0\right\}$$
containing the origin. \index{$\cK_f$}

For 
$A=(A_1,\dots,A_g)\in\mat{d\times e}^g$ and $P\in\mat{e\times \de}$, we write
\begin{alignat*}{3}
A^* & :=(A_1^*,\dots,A_g^*),\qquad &  \index{$A^*$}
A \mydot x &:= \sum_j^g A_jx_j, \\ \index{$\mydot$}
AP & :=(A_1P,\dots,A_gP),\qquad & \ker A &:= \bigcap_j^g  \index{$\ker A$}
\ker A_j.
\end{alignat*} 
For a hermitian monic  pencil $L=I-A\mydot x-A^*\mydot x^*$ set 
$\partial\cD_L(n)  = \cZ_L(n)\cap\cD_L(n)$
and
\[
\partial\cD_L=\bigcup_{n\in\N}\partial\cD_L(n). \index{$\partial \cD_L$}
\]
Observe that since $L(0)\succ0$, it is easy to see that $\partial\cD_L(n)$ is precisely the topological boundary of $
\cD_L(n)$. Furthermore, $\cD_L(n)$ is the closure of its interior because of convexity. 
A non-constant hermitian monic pencil  $L$ is \df{minimal} if it is of minimal 
size among  hermitian monic pencils $L'$ satisfying $\cD_{L'}=\cD_L$. 
If $L$ and $M$ are minimal and $\cD_L=\cD_M,$ then $L$ and $M$
 are unitarily equivalent. (See Proposition~\ref{p:mp}.)
It is convenient to declare that the minimal pencil for the largest 
free spectrahedron $\cD_I=\{(X,X^*)\colon X\in\mat{n}^n,n\in\N\}$ is of size $0$. 
Every free semialgebraic set strictly contained in $\cD_I$ is called {\bf proper}.

%---------------------

\subsection{Free loci and spectrahedra}
\label{sec:star2free}

For $h,n\in\N,$ let $\Omega^{(n)}=(\Omega^{(n)}_1,\dots,\Omega^{(n)}_h)$ be an $h$-tuple of $n\times n$ \df{generic matrices}, that is,
$$\Omega^{(n)}_j=(\omega_{j\imath\jmath})_{\imath\jmath},$$
where $\omega_{j\imath\jmath}$ for $1\le j\le h$ and $1\le\imath,\jmath\le n$ are commuting indeterminates.\index{$\Omega^{(n)}$}

\begin{lem}\label{l:irr}
A  linear pencil
$L=I-A\mydot z$
is \irrL if and only if
\ben[label={\rm(\arabic*)}]
\item
$\ker A=\{0\}$ and $\ker A^*=\{0\}$; and
\item
$\det L(\Omega^{(n)})$ is an irreducible polynomial for all $n$ large enough.
\een
\end{lem}

\begin{proof}
Assume $L$ is \irrL. Thus the $A_j$ have no common invariant subspace. In particular, $\ker A=\{0\}$ and $\ker A^*= \{0\}.$ Thus (1) holds.
The fact that (2) holds is contained in \cite[Theorem 3.4]{HKV}.

For the converse implication assume $L$ is not \irrL. So the $A_j$ have an invariant
subspace, and $L$ can be written in block form  as
\[
L=\begin{pmatrix}
L_1 & \star\\
0 & L_2
\end{pmatrix}.
\]
If the coefficients of $L_1$ are jointly nilpotent, then $\ker A\neq\{0\}$.
If the coefficients of $L_2$ are jointly nilpotent, then $\ker A^*\neq\{0\}$.
Otherwise $\det L_i(\Omega^{(n)})$ are non-constant for all large $n$
(cf.~Remark \ref{r:facts}\ref{it:rf4} below), and 
hence
\[
\det L(\Omega^{(n)})=\det L_1(\Omega^{(n)})\det L_2(\Omega^{(n)})
\]
is not irreducible for large $n$.
\end{proof}

Note that every \irrL hermitian monic pencil is minimal.
Two hermitian monic pencils $I_\de - \sum_{j=1}^g  A_j x_j  - \sum_{j=1}^g A_j^* x_j^*$ and $I_\de - \sum_{j=1}^g  B_j x_j  - \sum_{j=1}^g B_j^* x_j^*$
are unitarily equivalent if there is a unitary matrix $U$ such that
 $UA_j=B_jU$ for $1\le j\le g.$

\begin{prop}\label{p:mp} 
A minimal hermitian monic pencil is an orthogonal 
 direct sum of irredundant \irrL hermitian monic pencils.
 If $L_1$ and $L_2$ are minimal hermitian monic pencils
 with $\cD_{L_1}=\cD_{L_2},$ then $L_1$ and $L_2$ are unitarily
 equivalent.
\end{prop}

\begin{proof}
Let $L = I_\de - \sum_{j=1}^g  A_j x_j  - \sum_{j=1}^g A_j^* x_j^*$
be a given  hermitian monic pencil. By an invariant subspace
 for $L$, we mean an invariant subspace for $\{A_1,\dots,A_g,A_1^*,\dots,A_g^*\}.$
 Since $L$ is hermitian, any invariant subspace for $L$ 
 is in fact reducing.
Hence $L=\oplus_i L^i,$ where each $L^i$ is a hermitian monic
 pencil with no nontrivial  invariant (equivalently reducing) subspaces. 
 Thus each $L^i$ is an \irrL hermitian monic pencil.

 If there is an $i$ such that $\cD_{L^i}\subseteq \bigcap_{j\ne i} \cD_{L^j},$
 then, setting $M=\oplus_{j\ne i} L^j$ it follows that
 $\cD_M=\cD_L$ and $M$ has smaller size than $L.$ Hence if $L$
 is minimal, then $L$ is irredudant.

 The last statement is \cite[Theorem 1.2]{HKM}. See
 also \cite[Section 6]{DDOSS}. 
\end{proof}

\begin{prop}[{\cite[Proposition 8.3]{HKV}}]\label{p:pd-min}
If $L$ is a minimal hermitian monic pencil, then $\partial\cD_L(n)$ is Zariski dense in $\cZ_L(n)$ for all $n$ large enough.

In particular, if $f$ is a polynomial and $\partial \cD_L\subseteq \cZ_f$, 
then $\cZ_L\subseteq \cZ_f$.
\end{prop}

\begin{prop}\label{p:irrpoly}
If $f\in \matt{\de}{\pz}$ and  $\det f(0)\neq0$, 
then $f$ is an \irr if and only if $\det f(\Omega^{(n)})$ is an irreducible polynomial for all $n$ large enough.
\end{prop}

\begin{proof}
The forward implication  is
%$(\Rightarrow)$ is
\cite[Theorem 4.3(1)]{HKV}. For the converse, % $(\Leftarrow)$ 
suppose $f$ factors as $f=f_1f_2$, where the $f_i$ are non-invertible. 
By Remark \ref{r:facts}\ref{it:rf4} below, 
$\det f_i(\Omega^{(n)})$ is  non-constant for large $n$. But then 
$\det f(\Omega^{(n)})$ is not irreducible for large $n$.
\end{proof}

\begin{prop}\label{p:DZ}
Let $f\in\matt{\de}{\px}$ satisfy $\det f(0)\neq 0$, and let $L$ be a 
 hermitian monic pencil.
\begin{enumerate}[label={\rm(\arabic*)}]
	\item \label{i:DZ1} If $\cZ_f=\cZ_L,$ then $\cK_f=\cD_L$.
        \item \label{i:DZ2} If $L$ is minimal and $\cK_f=\cD_L$, then 
$\cZ_f \supseteq \cZ_L$.
	\item \label{i:DZ3} If $f$ is an \irr and $L$ is minimal, then $\cK_f=\cD_L$ implies $\cZ_f=\cZ_L$.
\end{enumerate}
\end{prop}

\begin{proof}
To prove item \ref{i:DZ1} let $(X,X^*)$ be a point in the connected component $\cO$ of
$$\left\{(X,X^*)\in\mat{n}^{2g}\colon \det f(X,X^*)\neq 0\right\}$$
containing the origin. Thus, there exists a path
$\gamma$ in $\cO$ with $\gamma(0)=0$ and $\gamma(1)=(X,X^*)$.
If  $L(X,X^*)\nsucc0,$ then there exists $t\in (0,1)$ such that $\det L(\gamma(t))=0$, contradicting $\cZ_f=\cZ_L$. Therefore $L(X,X^*)\succ0$. A similar argument shows
 $L(X,X^*)\succ0$ implies $(X,X^*)\in \cO$. Taking closures obtains $\cK_f=\cD_L$.

Taking up items \ref{i:DZ2} and \ref{i:DZ3}, suppose $L$ is minimal. If 
$\cK_f=\cD_L$, then  they have the same topological boundary.
Since the topological boundary of $\cK_f(n)$ is contained in $\cZ_f(n)$ and $\partial\cD_L(n)$ is Zariski dense in $\cZ_L(n)$ for large $n$ by Proposition \ref{p:pd-min}, $\cZ_f\supseteq\cZ_L$.
If also $f$ is \irr, then  $\cZ_f(n)$ is irreducible for large $n$ by Proposition 
\ref{p:irrpoly} and thus $\cZ_f=\cZ_L$.
\end{proof}

\subsection{Realization theory}\label{ssec:real}
Let $\matt{\de}{\rz}$  denote the $\de\times\de$ noncommutative (nc) rational functions in  $z_1,\dots,z_h$ \cite{Coh06,KVV12,Vol1}. \index{$\rz$}
Evaluations and  the involution for polynomials
naturally extend to 
$\matt{\de}{\rz}$ and 
$\matt{\de}{\rxx}$, respectively. Both operations are
entirely transparent for FM realizations (equation \eqref{e:fm}), the realizations we use in this paper.

\begin{rem}\label{r:facts}
 Realization theory in general has roots in automata theory \cite{Sch,Sch1,Fli,Fli1} and can be traced back further to \cite{Kleene}, while FM realizations in the commutative setting arise from control theory \cite{FM}.
For later use we recall the following fundamental results about minimal FM realizations. 
 Each is an embodiment of a well understood general principle of
 realization theory for matrix functions in one (commutative) variable \cite{BGKR}.
 For the original statements and proofs,  see \cite{BGM05,KVV09,Vol17}.

\begin{enumerate}[label={\rm(\arabic*)}]
	\item \label{it:rf1} 
      An FM  realization  $I+c^* (I-A\mydot z)^{-1}(b\mydot z)$ of size $d$
   with $c,b_j\in M_{d\times \delta}(\C)$      is \df{controllable} if
\[
  \Span\{A^wb_j u\colon w\in\Langle z\Rangle, 1\le j\le g, u \in \C^\de \}  =\C^d, \\
\]
 and \df{observable} if
\[
\Span\{(A^*)^wc u\colon w\in\Langle z\Rangle, u \in \C^\de \}  =\C^d.
\]
 It is a fundamental result that a realization  is
 minimal  if and only if it is observable and controllable.
 An immediate consequence, and one that is used here, is, for a minimal
 realization, 
	$$c^*v=0 \text{ and } v\in\ker A \quad \Rightarrow \quad v=0$$
	and
	$$v^*b=0 \text{ and } v\in\ker A^* \quad \Rightarrow \quad v=0.$$
	\item \label{it:rf2} 
   The state space isomorphism theorem says 
 minimal FM realizations are unique up to an isomorphism (change of basis)
	between their  state spaces. That is, if 
   $I+c^* (I-A\mydot z)^{-1}(b\mydot z)$ and $I+\gamma^* (I- B \mydot z)^{-1}(\beta \mydot z)$ 
    are two minimal FM realizations for the same rational function, then
 they have the same size, say $d,$ and there is $d\times d$ invertible
 matrix $S$ such that $SA=BS$, $Sb=\beta$ and $c^*S^{-1}=\gamma$.
   \item \label{it:rt+} Given a realization $I+c^* (I-A\mydot z)^{-1}(b\mydot z)$
  there is a linear algebra algorithm -- an extension of the Kalman decomposition -- that
 produces a minimal realization. 
	\item \label{it:rf3} In the classical (commutative) one-variable
 setting, if $\rr(\zeta)=I+ \zeta c^*(I-A \zeta)^{-1}b$ is a minimal FM realization, 
 then the domain of $\rr$ is precisely the set of $\zeta$ for which
  $I-\zeta A$ is invertible. In the present several variable noncommutative
 setting, while there are some subtleties in the statement of the analogous
 result found in \cite{KVV09,Vol17}, these results do justify calling
  the complement of $\cZ_L$ the \df{domain of regularity} of the rational 
 function with minimal realization  $\rr=I+c^*L^{-1}\bb.$    
	\item \label{it:rf4} If $\rr=I+c^*L^{-1}\bb$ is a minimal
 realization, then $\rr$ is a polynomial if and only if 
  the coefficients of $L$ are jointly nilpotent.
 Indeed, if $\rr$ is a polynomial, then  by item \ref{it:rf3}, $\cZ_L=\varnothing.$
 By \cite[Proposition 3.3]{KV}, $\cZ_L=\varnothing$
 if and only the coefficients of $L$ are jointly nilpotent. The converse
 is immediate.

\item \label{it:rf5}  Lastly, if $\rr=I+c^*L^{-1}\bb$ is an FM realization, then 
\begin{equation}\label{e:fmi}
\rr^{-1}=I-c^*\big(
I-(A-b c^*)\mydot z
\big)^{-1}\bb
\end{equation}
is an FM realization of $\rr^{-1}$ by \cite[Theorem 4.3]{BGM05}.
Because the realizations \eqref{e:fm} and \eqref{e:fmi} are of the same size, 
  \eqref{e:fm} is minimal for $\rr$ if and only if \eqref{e:fmi} is minimal for $\rr^{-1}$.
\end{enumerate}
\end{rem}

\begin{prop}\label{p:FM}
Let $f\in\matt{\de}{\pz}$ be non-constant with  $f(0)=I$. If $I+c^*L^{-1}\bb$ is a minimal FM realization of $f^{-1}$ with $L=I-A\mydot z$, then 
\begin{enumerate}[label={\rm(\arabic*)}]
	\item \label{FM1} $\det f(\Omega^{(n)})=\det L(\Omega^{(n)})$ for all $n$.
\end{enumerate}
If moreover $\de=1$, then
\begin{enumerate}[resume,label={\rm(\arabic*)}]
	\item \label{FM2} $\ker A^*=\{0\}$ and $\ker A=\{0\}$;
	\item  \label{it:indec} 
	$L$ is \irrL if and only if $f$ is an \irr.
\end{enumerate}
\end{prop}

\begin{proof}

(1) By the well-known determinantal identity $\det(M+uv^*)=\det(I+v^*M^{-1}u)\det M$ for an invertible $M$,
$$\det L(Z)\, \det f(Z)^{-1}
=\det\big( (L+\bb c^*)(Z) \big)
$$
for every $Z$ 
with $\det f(Z)\neq0$. By Remark \ref{r:facts}\ref{it:rf4}, 
$N_j:=A_j-b_jc^*,$ the coefficients of $L+\bb c^*,$ are the coefficients
in a minimal realization of the polynomial $f.$ By Remark \ref{r:facts}\ref{it:rf4}
, the $N_j$ are jointly nilpotent. 
Hence  $\det f(\Omega^{(n)})=\det L(\Omega^{(n)})$ for all $n$.

(2) 
If $0\neq v\in \ker A$, then 
\[
N_jv=-(c^*v)b_j,
\]
and $c^*v\in\C\setminus\{0\}$ by Remark \ref{r:facts}\ref{it:rf1}. Hence $b_j\in\ran N_j$.
Since the $N_j$ are jointly nilpotent, there exists a nonzero vector $u$ such that
$u^*N_j=0$. Hence $u^*b_j=0$. By Remark \ref{r:facts}\ref{it:rf1},
the FM realization \ref{e:fmi} is not minimal, contradicting
Remark \ref{r:facts}\ref{it:rf5}.

A similar line of reasoning shows that $\ker A^*=\{0\}$. 
If $v^*A_j=0$ and $N_ju=0$, then $-v^* b_j c^*u=0$. 
By minimality, there is a $k$ such that $v^* b_k\ne 0$. Hence
$c^*u=0$ and thus $A_ju=0$, contradicting minimality.

(3) Let $f$ be an \irr. By Proposition \ref{p:irrpoly}, $\det L(\Omega^{(n)})=\det f(\Omega^{(n)})$ is an irreducible polynomial for all $n$ large enough. 
Hence
$L$ is \irrL 
by Lemma \ref{l:irr} and (2).
Conversely,  if $L$ is \irrL, then $\det f(\Omega^{(n)})=\det L(\Omega^{(n)})$ is an irreducible polynomial for all $n$ large enough by Lemma \ref{l:irr}. 
Therefore $f$ is an \irr by Proposition \ref{p:irrpoly}.
\end{proof}

\section{Proof of Theorem \ref{thm:main}}
\label{sec:pfmain}

We start the proof of Theorem \ref{thm:main}
with a lemma.

\begin{lem}\label{l:exp}
Suppose  $\rr\in\rxx\setminus\C$ is defined at the origin and $\rr(0)=1$. Assume that $\rr$ is hermitian and $\rr=1+c^* L^{-1}\bb$ is a minimal FM realization, where $\bb=\sum_j \widecheck{b}_jx_j+\sum_j \widehat{b}_jx_j^*$. If $L$ is \irrL and monic hermitian, say $L= I- A\mydot x- A^*\mydot x^*$,
 then there exists $\lambda\in\R\setminus\{0\}$ such that 
$$\widecheck{b}_j=\lambda A_jc \ \ and \  
\ \widehat{b}_j= \lambda A_j^*  c \ \ \ \ 
 for \ all \ j=1,\dots,g.$$
\end{lem}

\begin{proof}
Since $\rr$ is hermitian, the comparison of formal power series expansions of $1+c^* L ^{-1} \bb$ and $1+\bb^* L ^{-1} c$ yields
\begin{align}
\label{e:1} c^* A_k\, w(A,A^*)\,\widecheck{b}_j & = \widehat{b}_k^*\, w(A,A^*)\,A_jc \\
\label{e:2} c^* A_k^*\, w(A,A^*)\,\widecheck{b}_j &= \widecheck{b}_k^*\, w(A,A^*)\,A_jc \\
\label{e:3} c^* A_k\, w(A,A^*)\,\widehat{b}_j &= \widehat{b}_k^*\, w(A,A^*)\,A_j^*c
\end{align}
for all $w\in\mxx$ and $1\le j,k\le g$. Since $L$ is \irrL, 
the matrices $w(A,A^*),$ for $w\in\mxx,$ span $\mat{d}$.
It is easy to see that if $v_1,v_2,v_3,v_4\in\C^d$ satisfy
$$v_1^*Mv_2=v_3^*Mv_4\qquad \text{for all }M\in\mat{d},$$
then $v_1$ and $v_3$ are collinear, and $v_2$ and $v_4$ are collinear. Hence by \eqref{e:1},\eqref{e:2},\eqref{e:3} and the fact that $w(A,A^*)$ span $\mat{d}$, there exist $\lambda_{jk}^1,\lambda_{jk}^2,\lambda_{jk}^3\in\C$ such that
\begin{alignat}{3}
\label{e:4} \widecheck{b}_j &=\lambda_{jk}^1 A_jc,&\qquad \widehat{b}_k&=\overline{\lambda_{jk}^1} A_k^*c \\
\nonumber
\widecheck{b}_j &=\lambda_{jk}^2 A_jc,&\qquad \widecheck{b}_k&=\overline{\lambda_{jk}^2} A_kc \\
\label{e:5}
\widehat{b}_j &=\lambda_{jk}^3 A_j^*c,&\qquad \widehat{b}_k&=\overline{\lambda_{jk}^3} A_k^*c
\end{alignat}
for all $j,k$. By minimality there exists $\ell$ such that $\widecheck{b}_{\ell}\neq0$ or $\widehat{b}_{\ell}\neq0.$  By symmetry we may assume $\widehat{b}_{\ell}\neq0$.

Since $\widehat{b}_\ell\ne 0$, equation \eqref{e:4} implies $\lambda:=\lambda_{j\ell}^1\ne 0$ 
is independent of $j$.  It also implies $A_\ell^* c\ne 0$.
Likewise, by equation \eqref{e:5}, $\lambda_{j\ell}^3$ is independent of $j$ and 
$\lambda_{j\ell}^3=\lambda_{j\ell}^1=\lambda$. 
By equation \eqref{e:5}, $\widehat{b}_\ell = \lambda A^*_\ell c = \overline{\lambda} A^*_\ell c$. Thus  $\lambda\in\R\setminus\{0\}$.
Finally,  from equation \eqref{e:4}, 
$\widecheck{b}_j =\lambda A_jc$ and $\widehat{b}_j = \lambda A_j^*c$
as desired.
\end{proof}

\begin{prop}
\label{p:fL}
 Suppose $f\in\px$ is a hermitian atom and $f^{-1}=1+c^*L^{-1}\bb$ is a minimal FM realization.
 If $L$ is hermitian, then $f$ is concave, has degree at most two
 and is a Schur complement of $L$. Further, $f(X,X^*)\succ 0$ if and only if $L(X,X^*)\succ 0$.
\end{prop}

\begin{proof}
Since $L$ is hermitian, it has the form $L=I-A\mydot x - A^* \mydot x^*$.
Since $f$ is an atom and the realization $f^{-1}=I+c^* L^{-1}\bb$ is minimal,
 $L$ is \irrL by Proposition \ref{p:FM}\ref{it:indec}. Since $f$ is hermitian,
 so is $f^{-1}$. Thus,  by Lemma \ref{l:exp}  we may assume that
$$\bb= \ve(A\mydot x+A^*\mydot x^*)c$$
for $\ve\in\{-1,1\}$. By Remark \ref{r:facts}\ref{it:rf5}, $f$ admits a minimal realization
\begin{equation}\label{e:real2}
f=1 - \ve c^*\Big( I-A(I-\ve cc^*)\mydot x-A^*(I-\ve cc^*)\mydot x^* \Big)^{-1}(A\mydot x+A^*\mydot x^*)c.
\end{equation}
Since $f$ is a polynomial, the $A_j(I-\ve cc^*),A_j^*(I-\ve cc^*)$ are jointly nilpotent by Remark \ref{r:facts}\ref{it:rf4}. In particular, they have a nontrivial common kernel. Since $A_j,A_j^*$ generate $\mat{d}$, it follows that $P=I-\ve cc^*$ is singular, so in particular $\ve=1$. Since also $P$ is  hermitian 
and a rank-one perturbation of the identity, it is an orthogonal projection. 
After a unitary change of basis we assume that $P=0\oplus I_{d-1}$. Let
$$A=\begin{pmatrix} \alpha & v^* \\ u & \tA \end{pmatrix}$$
be the decomposition of $A$ with respect to this new basis. Then
$$AP=\begin{pmatrix} 0 & v^* \\ 0 & \tA \end{pmatrix},\qquad A^*P=\begin{pmatrix} 0 & u^* \\ 0 & \tA^* \end{pmatrix}$$
are jointly nilpotent, so $\tA,\tA^*$ are jointly nilpotent. Hence
$\tA_j^* \tA_j$ is nilpotent and thus  $\tA=0$.
It follows that  $AP,A^*P$ are jointly nilpotent of order at most two and
$$\Big(I-A(I-cc^*)\mydot x - A^*(I-cc^*)\mydot x^*\Big)^{-1}
=I+A(I-cc^*)\mydot x + A^*(I-cc^*)\mydot x^*.$$
Now \eqref{e:real2} gives
\begin{align*}
f
&= 1 - c^*\Big(I+A(I-cc^*)\mydot x+A^*(I-cc^*)\mydot x^*\Big)\big(A \mydot x+A^* \mydot x^*\big)c \\
&= 1 - c^*(A\mydot x+A^*\mydot x^*)c-c^*(A\mydot x+A^*\mydot x^*)(I-cc^*)(A\mydot x+A^*\mydot x^*)c.
\end{align*}
Therefore $f$ has the form
$$f=1 -(\alpha\mydot x+\bar{\alpha} \mydot x^*)-(u\mydot x+v\mydot x^*)^*(u\mydot x+v\mydot x^*),$$
which is a Schur complement of
\[
L=I-\begin{pmatrix} \alpha & v^* \\ u & 0 \end{pmatrix}\mydot x
-\begin{pmatrix} \bar{\alpha} & u^* \\ v & 0 \end{pmatrix}\mydot x^*.
\]
In particular, $f$ is concave, has degree at most two and $f(X,X^*)\succ 0$ if and only if $L(X,X^*)\succ 0$.
\end{proof}

\begin{prop}
\label{p:DfDL}
  Suppose $f\in\px$ is a hermitian atom with $f(0)=1$ and $L$ is a minimal 
  hermitian monic pencil of size $d\ge 1$.
  If $\cD_f=\cD_L$, then $L$ is \irrL and
there exists $b_j,c\in \C^d$ such that $f^{-1}=I+ c^* L^{-1}\bb$ 
 is a minimal FM realization.
\end{prop}

\begin{proof}
Write
$L=I-A\mydot x-A^*\mydot x^*.$  By Proposition \ref{p:DZ}\ref{i:DZ3}, $\cZ_f=\cZ_L$. After a unitary change of basis we can assume that $L$ equals a direct sum of \irrL hermitian monic pencils $L^1,\dots,L^\ell$. Since $L$ is minimal, the pencils $L^1,\dots,L^\ell$ are pairwise unitarily non-similar by Proposition \ref{p:mp}. Therefore
\[
 \cZ_f(n)=\cZ_L(n)=\cZ_{L^1}(n)\cup\cdots\cup \cZ_{L^\ell}(n)
\]
is a union of $\ell$ distinct hypersurfaces for large $n$ by Lemma \ref{l:irr}. 
Since $f$ is an atom, Proposition \ref{p:FM}\ref{it:indec} implies $\ell=1.$
 Hence $L$ is \irrL. 

Let $f^{-1} = 1+\tc^* \tL ^{-1} \widetilde{\bb}$ be a minimal FM realization. Since $f$ is an atom, $\tL$ is \irrL by Proposition \ref{p:FM}\ref{it:indec},
and $\cZ_{\tL}=\cZ_f=\cZ_{L}$ by Proposition \ref{p:FM}\ref{FM1}. By \cite[Theorem 3.11]{KV}, the pencils $L$ and $\tL$ are of the same size $d$ and there exists $P\in\operatorname{GL}_d(\C)$ such that $\tL=P^{-1}LP$. Therefore $f^{-1}$ admits the minimal FM realization
\[
f^{-1}=1+c^* L ^{-1} \bb,
\]
where $\bb=P\widetilde{\bb}$ and $c=P^{-*}\tc$.
\end{proof}

Combining Propositions \ref{p:DfDL} and \ref{p:fL},
 and using  the fact that if $\cD_f$ is convex, then
 there is a minimal  hermitian monic pencil $L$
 such that $\cD_f=\cD_L$ \cite{HM12}, proves a bit
more than claimed in  Theorem \ref{thm:main}. 

\begin{cor}
\label{c:main+}
Suppose $f\in\px$ is a hermitian atom  and  $f(0)\succ 0$. 
If $\cD_f$ is proper and convex, then $f$ has degree two and is concave.

Further, normalizing $f(0)=1$, if  $L$ is a minimal 
hermitian monic pencil such that $\cD_f=\cD_L$, 
then $L$ is \irrL, $f$ is a Schur complement of $L$ and there
exist vectors $c,b_1,\dots, b_{2g}$ such that 
\[
 f^{-1} = 1+ c^* L^{-1}\bb
\]
is a minimal FM realization.
\end{cor}

\begin{rem}
The properness in Corollary \ref{c:main+} ensures that a minimal hermitian monic pencil for $\cD_f$ has size at least 1, so Proposition \ref{p:DfDL} applies. For the description of $f\in\px$ satisfying $f\succ0$ globally, see \cite[Remark 5.1]{KPV17}.	\qed
\end{rem}

\begin{rem}
	From the proof of Theorem \ref{thm:main}  we also obtain a bound on $d$, the size of $L$.
	Since $\tA=0$, the lower right $(d-1)\times (d-1)$ entries in the $\C$-algebra
	generated by $A$ and $A^*$ are spanned by $S=\{s t^*: s,t\in \{u_1,\dots,u_g,v_1,\dots,v_g\}\}$.
	Since $L$ is \irrL, this span is all of $\mat{d-1}$ and hence $(d-1)^2$ is
	at most the maximal cardinality of $S$, namely, $(2g)^2$. Hence  $d\le 2g+1$.
\qed
\end{rem}

\section{Proof of Theorem \ref{thm:main2} and algorithms: Corollary \ref{cor:algo12}}\label{sec:algo}

In this section we prove Theorem \ref{thm:main2} and explore algorithmic consequences.
In particular, we present, stated as Corollary \ref{cor:algo12},
 a constructive version of the main result of \cite{HM12}. 

\subsection{Proof of Theorem \ref{thm:main2}}\label{ssec:pf12}
It suffices to prove item \ref{it:spec} implies item \ref{it:midInv}.
Let $L$ be the pencil appearing in a minimal FM realization for $f^{-1}$, and let $L^1,\dots,L^\ell$ be its diagonal blocks as in \eqref{e:triang}. By Remark \ref{r:facts}\ref{it:rf3},
$\cK_f=\cK_L$. By assumption there exists a minimal hermitian monic pencil $\tL$ such that $\cK_L=\cD_{\tL}$. By $\partial \cK_L(n)$ we denote the topological boundary of $\cK_L(n)$. Thus
$$\cZ_L(n)\supseteq\partial \cK_L(n)=\partial \cD_{\tL}(n)$$
for every $n$.

For $S\subseteq \mat{n}^g$ let $\zar{S}$ denote its Zariski closure. For $n$ sufficiently large,
$$\cZ_L(n)\supseteq\zar{\partial \cK_L(n)}=\zar{\partial \cD_{\tL}(n)}=\cZ_{\tL}(n)$$
by Proposition \ref{p:pd-min}. Note that $\cZ_{L}(n)$ and $\cZ_{\tL}(n)$ are hypersurfaces. Therefore the set of irreducible components of $\cZ_{L}(n)$ contains the set of irreducible components of $\cZ_{\tL}(n)$. Since
$$\cZ_L=\cZ_{L^1}\cup\cdots\cup \cZ_{L^\ell}$$
and the $\cZ_{L^i}(n)$ are 
irreducible hypersurfaces for all $n$ large enough by Lemma \ref{l:irr}, there exist indices $1\le i_1<\cdots<i_s\le \ell$ such that
the $L^{i_k}$ are pairwise non-similar and
\begin{equation}\label{e:zar}
\zar{\partial \cK_L(n)}=\cZ_{\tL}(n) =\cZ_{L^{i_1}}(n)\cup\cdots \cup\cZ_{L^{i_s}}(n)
\end{equation}
for all $n$ large enough.
Since $\tL$ is minimal, it is (up to a unitary change of basis) equal to a direct sum of irredundant \irrL hermitian monic pencils $\tL^k$ by Proposition \ref{p:mp}. Each of them corresponds to an irreducible component in \eqref{e:zar} by Proposition \ref{p:pd-min}. Therefore $\tL= \tL^1\oplus\cdots\oplus \tL^s$ and,
after reindexing if needed,  $\cZ_{\tL^k}=\cZ_{L^{i_k}}$ for $k=1,\dots,s$. Then $\cK_{L^{i_k}}=\cD_{\tL^k}$ is convex for every $k$ and therefore
\begin{equation}\label{e:lmi}
\cK_L=\bigcap_k\cK_{L^{i_k}}=\bigcap_k\cD_{\tL^k}=\cD_{\tL^1\oplus\cdots\oplus \tL^s}.
\end{equation}
Moreover, $L^{i_k}$ is similar to $\tL^k$ by \cite[Theorem 3.11]{KV}.

Recall that $\hL$ is the direct sum of \irrL blocks $L^k$ that are similar to a hermitian monic pencil, and $\vL$ is the direct sum of the rest. Then every $L^{i_k}$ appears as a direct summand in $\hL$. Now let $L^m$ be an arbitrary pencil appearing in $\hL$. If it is not similar to $L^{i_k}$ for any $k$, then \eqref{e:zar} implies
$$\bigcap_k\cK_{L^{i_k}}\subseteq \cK_{L^m}.$$
Hence $\cK_f=\cD_{\hL}$ holds by \eqref{e:lmi}.
\qed

\begin{rem}\label{rem:sacvx}
Given a factorization of $f$ into atomic factors $f=f_1\cdots f_t$ with $f_j(0)=I$, one can use the proof of Theorem \ref{thm:main2} to 
identify those factors $f_j$ that  determine $\cK_f$.

By \eqref{e:zar}, 
$$\cZ_{L^{i_1}}(n)\cup\cdots \cup\cZ_{L^{i_s}}(n)\subseteq 
\cZ_{f_1}(n)\cup\cdots \cup\cZ_{f_t}(n).$$
for all $n$. Since $\cZ_{f_j}(n)$ is an irreducible surface for large $n$ by Proposition \ref{p:irrpoly}, there exist indices $1\le j_1<\cdots<j_s\le t$ such that
 $$\cZ_{L^{i_k}}=\cZ_{f_{j_k}}$$
for all $k=1,\dots,s$. Therefore
$$\cK_f=\bigcap_k\cK_{f_{j_k}}$$
by \eqref{e:lmi} and Proposition \ref{p:DZ}\ref{i:DZ1}. 

To find the indices $j_k$, we first compute minimal realizations for $f_j^{-1}=I +c_jL_j^{-1}\bb_j$, and put each $L_j$ into a block upper triangular form as in \eqref{e:triang}. For every $j$, precisely one of the blocks on the diagonal of $L_j$ is \irrL by Proposition \ref{p:irrpoly}. Then we compare these blocks to the pencils $L^{i_k}$ to determine $j_k$.
\qed
\end{rem}

\begin{proof}[Proof of Corollary \ref{cor:sacvx}]
$(\Leftarrow)$ is trivial. For the converse let $f=\prod_i f_i$ and consider a minimal FM realization $f^{-1}=I+c^*L^{-1}\bb$.
After a basis change 
we may assume that
$L$ is of the form \eqref{e:triang}. 
As in Remark \ref{rem:sacvx}, for every $i$ there
exists $j_i$ such that $\cZ_{L^i}=\cZ_{f_{j_i}}$, whence $\cK_{L^i}=\cK_{f_{j_i}}$. If some $L^i$
is not similar to a hermitian monic pencil, then 
 $\vL$ 
is nontrivial and is invertible
on $\interior \cK_{\hL}$ by 
convexity of $\cK$ and
Theorem \ref{thm:main2}. 
Hence $f_{j_i}$ is redundant, contradicting the assumption.
\end{proof}

\subsection{Finding an LMI representation for a convex $\cK_f$}\label{ssec:algo}
The main result of \cite{HM12} states that for
a hermitian  matrix polynomial $f\in\matt{\de}{\px}$ with $f(0)\succ0$, the set
$\cK_f(n)$ is convex for all $n$ if and only if $\cK_f$ is a free spectrahedron.
Actually, the version in \cite{HM12} does this for hermitian $f$ with
bounded $\cK_f$. However, these two assumptions are redundant. Indeed, the former
can be enforced by replacing $f$ by $f^*f$. The alternative proof of
\cite[Theorem 1.4]{HM12} due to Kriel \cite{Kri} is based on
Nash functions in real algebraic geometry and
the Fritz-Netzer-Thom  characterization \cite{FNT} of free spectrahedra
via operator systems theory.
It also works for unbounded $\cK_f=\cD_{f^*f}$.

\subsubsection{Algorithm}\label{sssec:algo}
We next explain how the machinery developed in this paper produces an
explicit 
minimal 
LMI representation for a convex $\cK_f$.
This efficient algorithm only  involves linear
algebra and semidefinite programming (SDP) \cite{WSV,BPR13}.

\begin{enumerate}[label={\rm(\alph*)}]
\item
Compute the minimal realization
\[
I+c^*L^{-1}\bb
\]
for $f^{-1}$.
To construct this realization, one uses the explicit state-space formulae, c.f. \cite[Section 4]{BGM05},
for addition and multiplication to construct a realization for $f$, applying the 
Kalman decomposition \cite[Section 7]{BGM05} at each step to ensure minimality. 
Lastly, the formula for inversion \eqref{e:fmi} yields a minimal realization 
for $f^{-1}$. This process only uses linear algebra, and minimization after 
every step keeps the sizes of intermediate realizations from blowing up.
 {\it Mathematica} notebooks with rudimentary programs for computing minimal realizations are found in \cite{HdOMS}.
\item
Next we find the Burnside decomposition \cite[Corollary 5.23]{Bre} of $L$ into 
\[
L=\begin{pmatrix}
{L}^1 & \star & \star \\
 & \ddots &\star \\
 & & {L}^\ell 
\end{pmatrix},
\]
where each $L^i$ is either \irrL or the identity.
This decomposition can be found using deterministic algorithms with polynomial time complexity. Let $\cA$ be the unital matrix subalgebra generated by the coefficients of $L$. One first computes and mods out the radical of $\cA$ (corresponding to the $\star$ entries) using the algorithm in \cite[Section 3]{dGIKR}; then the
algorithm of \cite[Theorem 3.5]{Ebe} is applied to find the irreducible blocks $L^j$. Alternatively, \cite[Theorem 6]{CIK} gives an algorithm for decomposing $\cA$ as a direct sum of minimal left ideals; after omitting the ideals contained in the radical using a linear test \cite[Corollary 4.3]{FR}, the remaining ideals are necessarily one-dimensional, and the union of bases of ranges of their generators is a basis in which $L$ has the desired block structure.

\item
Considering only the \irrL blocks, choose one from each similarity class.
 Note that checking similarity of linear pencils amounts to checking whether the
system of linear equations 
$PL^i=L^jP$ has an invertible solution $P$. 
\item
Find all those $L^i$ that are similar to a hermitian monic pencil. 
This uses SDP. Each solution to the feasibility semidefinite problem
\beq\label{e:sdp}
Q\succeq I, \qquad Q(L^i)^*=L^iQ
\eeq
leads to a hermitian monic pencil 
$\tL^i=Q^{-\frac12}L^iQ^{\frac12}$.
If \eqref{e:sdp} is infeasible, then $L^i$ is not similar to a hermitian monic pencil.
\item
The direct sum $\tL$ of the hermitian monic pencils $\tL^i$ obtained in (d) satisfies
\[
\cD_{\tL}=\cK_f
\]
by Theorem \ref{thm:main2}.
\item
Using the minimization algorithm described in \cite[Subsection 4.6]{HKM}, which uses SDP to 
eliminate redundant blocks in $\tL$, we can produce a minimal hermitian monic pencil  $\hL$ with
$\cD_{\hL}=\cK_f$.
\end{enumerate}

\subsection{Checking whether  $\cK_f$ is convex} \label{ssec:algo2} 

As a side product of Theorem \ref{thm:main2}
and the Algorithm in Subsection \ref{ssec:algo}
we obtain a procedure for checking whether $\cK_f$ is convex. 

Given $f\in\matt{\de}{\px}$ with $f(0)=I$, 
we construct the realization of $f^{-1}$ 
and identify its \irrL blocks $L^i$, choosing one from
each similarity class.
Let $\hL$ be the direct sum of all the 
$L^i$ that are similar to a hermitian monic pencil, and let $\vL$ be the direct sum of the others. By Theorem \ref{thm:main2},
it  suffices to 
present an algorithm for checking whether property
\ref{it:lastInv} of Theorem \ref{thm:main2}
holds, that is, whether $\vL$ is invertible on the interior of $\cD_{\hL}$. To this end we first prove  general statements about (rectangular) affine linear pencils being of full rank on the interior of a free spectrahedron (see also \cite{KPV17,Vol+,Pas18,GGOW16} for related results). 

For the rest of this section let $L$ be a $d\times d$ hermitian monic pencil, and let $\tL$ be a $\delta\times \ve$ affine linear pencil (in $x$ and $x^*$). Assume $\de\geq\ve$ and consider the following system:
\begin{equation}\label{eq:preSDP}
\real(D\tL) = P_0 + \sum_k C_k^* L C_k,\quad P_0\succeq0
\end{equation}
for some 
$D\in \mat{\ve\times \de}$,
$C_k\in \mat{d\times \ve}$ and $P_0\in\mat{\ve}$, where $\real (M)=\frac12(M+M^*)$ denotes the real part of a square matrix $M$. 
(If $\de<\ve$ we simply replace $\tL$ by $\tL^*$.)
Note that
 $D=0$,  $P_0=0$, $C_k=0$ is a trivial solution. 
We mention that 
 \eqref{eq:preSDP} is  related to the
notion of a $\tL$-real left module of \cite{HKN}.

\begin{lem}\label{l:yes}
Let $\de\geq\ve$. If there exists a solution of \eqref{eq:preSDP} satisfying
\begin{equation}
\label{e:yes}
  \ker P_0 \cap \bigcap_k \ker C_k=\{0\},
\end{equation}
then $\tL(X,X^*)$ is full rank for every $X$ satisfying $L(X,X^*)\succ0$.
\end{lem}

\begin{proof}
Suppose  \eqref{eq:preSDP} holds and 
 $X\in\mat{n}^g$ satisfies $L(X,X^*)\succ0$. If $\real(D\tL)(X,X^*)v=0$ for $v\in\C^{\ve n}$, then \eqref{eq:preSDP} together with $P_0\succeq0$ and $L(X,X^*)\succ0$ imply
$$(P_0\otimes I)v=0,\quad (C_k\otimes I)v=0\quad \text{for all }k.$$
Therefore $v=0$ by equation \eqref{e:yes}. Hence $\real(D\tL)(X,X^*)$ is positive definite, so $(D\tL)(X,X^*)$ is invertible. Consequently $\tL(X,X^*)$ has full rank.
\end{proof}

\begin{prop}\label{p:no}
Let $\de\geq\ve$. If every solution  of \eqref{eq:preSDP} satisfies
$$P_0=0,\quad C_k=0\quad \text{for all }k,$$
then there exists $X\in \mat{\max\{d,\ve\}}^g$ such that $L(X,X^*)\succ0$ and $\ker\tL(X,X^*)\neq \{0\}$.
\end{prop}

Before proving Proposition \ref{p:no} we introduce some  notation. Let $\eta=\max\{d,\ve\}$. For $\ell=0,1,2,$ let $\cV_\ell\subseteq \matt{\eta}{\px}$ 
denote  the subspace of polynomials of degree at most $\ell$, and let
\begin{align*}
\cS &= \left\{\sum_i L_i^*L_i\colon L_i\in \cV_1 \right\}, \\
\cC &= \left\{\sum_k C^*_kLC_k\colon C_k\in \mat{d\times \eta} \right\}, \\
\cU &= \left\{
\begin{pmatrix}D_1\tL+\tL^*E_1^* & \tL^*E_2^* \\ D_2\tL & 0\end{pmatrix}
\colon D_1,E_1\in \mat{\ve\times \delta},\, D_2,E_2\in \mat{(\eta-\ve)\times\delta} \right\}.
\end{align*}
Also let $\cVh_2\subseteq \cV_2$ be the $\R$-subspace of hermitian matrix polynomials. 
Both $\cC$ and $\cS$ are convex cones in $\cVh_2$, and $\cU$ is a subspace in $\cV_2$. Observe that 
$$\cU\cap\cVh_2 = \left\{
\begin{pmatrix}\real(D_1\tL) & \tL^*D_2^* \\ D_2\tL & 0\end{pmatrix}
\colon D_1\in \mat{\ve\times \delta},\, D_2\in \mat{(\eta-\ve)\times\delta} \right\}$$
and $\cU=(\cU\cap\cVh_2)+i(\cU\cap\cVh_2)$. Using the standard argument involving Caratheodory's theorem on convex hulls \cite[Theorem 2.3]{Bar} it is easy to show that $\cC+\cS$ is closed in $\cVh_2$; see e.g. \cite[Proposition 3.1]{HKM12}. 

\begin{lem}\label{l:int}
	Keep the notation from above. If every solution  of \eqref{eq:preSDP} satisfies
$$P_0=0,\quad C_k=0\quad \text{for all }k,$$
then $\cU\cap(\cC+\cS)=\{0\}$.
\end{lem}

\begin{proof}
Suppose
\begin{equation}\label{e:44}
\begin{pmatrix}\real(D_1\tL) & \tL^*D_2^* \\ D_2\tL & 0\end{pmatrix}
= \sum_i L_i^*L_i
+\sum_k \begin{pmatrix}C_k^* \\ C_k'^*\end{pmatrix}L\begin{pmatrix}C_k & C_k'\end{pmatrix}
\end{equation}
for $D_1\in \mat{\ve\times \delta}$, $D_2\in \mat{(\eta-\ve)\times\delta}$, $L_i\in \cV_1$, $C_k\in\mat{d\times\ve}$ and $C'_k\in\mat{d\times(\eta-\ve)}$. By looking at the degrees on both sides we obtain $L_i\in \cV_0$; let us write
$$\sum_i L_i^*L_i=\begin{pmatrix}p_1 & p_2 \\ p_2^* & p_3\end{pmatrix} .$$
Therefore $\real(D_1\tL)$ satisfies \eqref{eq:preSDP}, so $p_1=0$ and $C_k=0$ by the hypothesis. Moreover, $p_2=0$ by positive semidefiniteness. Finally, since $L$ is monic, \eqref{e:44} implies $p_3=0$ and $C'_k=0$.
\end{proof}

To prove Proposition \ref{p:no} we require a version of the Gelfand-Naimark-Segal (GNS) construction.
 Given a Hilbert space $H$, let $\EndH(H)$ denote the (bounded linear) operators on $H$.

\begin{lemma}
\label{l:aY}
Suppose $\lambda:\cV_2\to \C$ is a positive linear functional in the sense that $\lambda(f^*f)>0$ for
all $f\in \cV_1\setminus\{0\}$.  
Thus, the resulting scalar product  $\langle f_1,f_2\rangle_\lambda :=\lambda(f_2^*f_1)$
on $\cV_1$ makes $\cV_1$ a Hilbert space and $\cV_0\subseteq \cV_1$
is a subspace. Let  $\pi:\cV_1\to \cV_0=\mat{\eta}$ denote the orthogonal projection.  For $a\in\mat{\eta}$ let $\ell_a\in\EndH(\cV_0)$ denote
the map $f\mapsto af$, and let $Y_j\in\EndH(\cV_0)$ denote the map
$f\mapsto \pi(x_jf)$. Then,
\begin{enumerate}[label={\rm(\arabic*)}]
	\item $\ell_a^*=\ell_{a^*}$; 
	\item $Y_j^*f =\pi(x_j^* f)$;
	\item $\ell_a Y_j = Y_j \ell_a$ (and hence $\ell_aY_j^*=Y_j^* \ell_a$);
	\item \label{it:Ustar} 
       there is a unitary mapping $U: \C^\eta\otimes \C^\eta\to \cV_0$ such that
	$U^*\ell_a U = a\otimes I$; 
	\item \label{it:aY-last}
             there exists $X_j\in \mat{\eta}$ such that $U^*Y_j U = I\otimes X_j$, and 
	 if $L=C+ \sum_j A_j x_j +\sum_j B_j x_j^*$ 
        is an affine linear pencil of
	size $\eta$,  then 
	\[
	U L(X,X^*)U^* = \ell_C+\sum_j \ell_{A_j} Y_j + \sum_j \ell_{B_j} Y_j^*.
	\]
\end{enumerate}
\end{lemma}

\def\ti{{\rm t}}

\begin{proof}
The proofs of the first three items are straightforward. To prove \ref{it:Ustar}, 
since $\lambda|_{\cV_0}$ is a linear functional on $\mat{\eta}=\cV_0$, there
is a matrix $P\in \mat{\eta}$ such that $\lambda(f)=\trace(Pf)$. 
Further, since $\lambda$ is positive, $P$ is positive definite.
Define $U$ by $U(u\otimes v) = uv^\ti P^{-\frac 12}$ 
and extend by linearity. By the definition of $\langle\cdot,\cdot\rangle_\lambda$, 
\[
\begin{split}
\langle U (u_1\otimes v_1), U (u_2\otimes v_2)\rangle_\lambda 
& =  \lambda\left((u_2v_2^\ti P^{-\frac12})^* u_1v_1^\ti P^{-\frac 12} \right)\\
& =  \trace \left((u_1v_1^\ti P^{-\frac 12}) P (P^{-\frac 12}  (u_2v_2^\ti)^*) \right)\\
& =  \trace(u_1v_1^\ti (v_2^*)^\ti  u_2^*) =  \langle u_1,u_2\rangle \, \langle v_1,v_2\rangle,
\end{split}
\]
so $U$ is unitary.   Similarly, for $a\in \mat\eta$,
\[
\begin{split}
\langle U^* \ell_a U (u_1\otimes v_1),(u_2\otimes v_2)\rangle_\lambda
&=  \trace\Big(((au_1)v_1^\ti P^{-\frac 12}) P (P^{-\frac 12}  (u_2v_2^\ti)^* )\Big) \\
&=  \langle au_1,u_2\rangle \, \langle v_1,v_2\rangle \\
&=  \langle (a\otimes I)(u_1\otimes v_1), u_2\otimes v_2\rangle.
\end{split}
\]

Since $Y_j$ commutes with each $\ell_a$, it follows that $U^* Y_j U$ commutes
with each $a\otimes I$. Hence there is a $X_j\in \mat\eta$ such that
$U^*Y_j U = I\otimes X_j$ and hence $U^*Y_j^* U = I\otimes X^*_j$.
Finally, observe that 
\[
A_j\otimes X_j = (A_j\otimes I)(I\otimes X_j) = U^* \ell_{A_j} Y_j U
\]
and analogously $B_j\otimes X_j^*=U^* \ell_{B_j} Y_j^* U$.
\end{proof}

\begin{proof}[Proof of Proposition \ref{p:no}]
By Lemma \ref{l:int}, $\cU\cap(\cC+\cS)=\{0\}$. Since  $\cC+\cS$ is also closed and convex
and since $\cU$ is a subspace,   by \cite[Theorem 2.5]{Kle} there exists an $\R$-linear functional $\lambda_0:\cVh_2\to \R$ satisfying
$$\lambda_0\left((\cC+\cS)\setminus\{0\}\right)=\R_{>0},\qquad \lambda_0(\cU\cap \cVh_2)=\{0\}.$$
We extend $\lambda_0$ to $\lambda:\cV_2\to\C$ by
$$\lambda(f)= \lambda_0\left(\frac12(f+f^*)\right)+i\lambda_0\left(\frac{1}{2i}(f-f^*)\right).$$
Note that $\lambda$  vanishes on $\cU$. Since $\lambda(\cS\setminus\{0\})=\R_{>0}$, $\lambda$ is a positive functional, so Lemma \ref{l:aY} applies; we assume the notation therein.

Write $\tL=\tC+ \sum \tA_j x_j +\sum \tB_j x_j^*$ for $\tC,\tA_j,\tB_j\in\mat{\delta\times\ve}$. 
For $D\in\mat{\eta\times (\delta+\eta-\ve)}$, let\looseness=-1 
\begin{equation}\label{e:FD}
F_D := U (D\, (\tL\oplus I_{\eta-\ve})(X,X^*)) U^*
 = \ell_{D(\tC\oplus I)}+\sum_j \ell_{D(\tA_j\oplus 0)} Y_j+\sum_j \ell_{D(\tB_j\oplus 0)} Y_j^*;
\end{equation}
the second equality in \eqref{e:FD} holds by Lemma \ref{l:aY}\ref{it:aY-last}. Let $u$ denote $I_{\ve}\oplus 0\in\mat{\eta}$ 
considered as a vector in $\cV_0$. Then
\begin{align*}
F_D u
& = \left(\ell_{D(\tC\oplus I)}+\sum_j \ell_{D(\tA_j\oplus 0)} Y_j+\sum_j \ell_{D(\tB_j\oplus 0)} Y_j^*\right)u \\
& = \pi\left(
D(\tC\oplus I)(I\oplus 0)+\sum_j D(\tA_j\oplus 0)(I\oplus 0)x_j+\sum_j D(\tB_j\oplus 0)(I\oplus 0)x_j^*
\right) \\
&=\pi(D(\tL\oplus 0)).
\end{align*}
Hence for every  $f\in \cV_0$,
$$\langle F_D u,f\rangle_\lambda=  \langle D(\tL\oplus 0),f\rangle_\lambda=   \lambda (f^*D (\tL\oplus 0))=0,$$
since $f^*D (\tL\oplus 0)\in \cU$. Thus $F_Du=0$ for all $D\in\mat{\eta\times (\delta+\eta-\ve)}$. Consequently $$(\tL\oplus I)(X,X^*)U^*u=0$$
and hence $\ker \tL(X,X^*)\neq\{0\}$.

Now fix $0\ne v\in \cV_0=\mat{\eta}$ and choose an isometry 
$V:\C^d\to \C^\eta$  such that $V^*v\ne 0$. If $L=I+\sum_jA_jx_j+\sum_jA_j^*x_j^*$, then
$$ U((V\otimes I)L(X,X^*)(V^*\otimes I))U^*=\ell_{VV^*}+\sum_j \ell_{VA_jV^*} Y_j+\sum_j \ell_{VA_j^*V^*} Y_j^*$$
by Lemma \ref{l:aY}\ref{it:aY-last} and thus
\[
\begin{split}
\langle U((V\otimes I)L(X,X^*)(V^*\otimes I))U^* v,v\rangle_\lambda
&=  \langle \pi (VLV^* v),v\rangle_\lambda
=  \lambda (v^*V L V^*v) >0
\end{split}
\]
since $v^*V L V^*v\in\cC$ is nonzero. It follows that $L(X,X^*)$ is positive definite.
\end{proof}

\begin{cor}\label{c:rank}
Let $L$ be a $d\times d$ hermitian monic pencil. If
 $\tL$ is a $\de\times\ve$ affine linear pencil
 such that $\tL(X,X^*)$ is full rank for every $X$ in the interior of $\cD_L(\max\{d,\delta,\ve\})$, 
then $\tL$ is full rank on the interior of $\cD_L$.
\end{cor}

The proof of Corollary \ref{c:rank} given below, while not the most efficient, yields
an algorithm presented in Subsection \ref{sss:algo3} below.

\begin{proof}
Without loss of generality, suppose $\delta\ge \ve$ 
 and let $\sigma=\max\{d,\delta\}$.

Given $\eta\le \delta$ and $\tL,$ an affine linear pencil of size $\de\times\eta$ 
such that $\tL(X,X^*)$ is full rank for each $X$ in the interior of $\cD_L(\sigma)$,
consider solutions  to the system \eqref{eq:preSDP}, i.e.,
\begin{equation}\label{e:45}
\real(D\tL)=P_0 + \sum_k C_k^* L C_k,\quad P_0\succeq0,
\end{equation}
and denote $V=\ker P_0 \cap \bigcap_k \ker C_k\subseteq \C^{\eta}$.
If, for each solution,  $V=\C^{\eta}$ 
(equivalently $P_0=0$, $C_k=0$), then there exists $X\in\mat{\sigma}^g$ such 
that $L(X,X^*)\succ0$ and $\ker\tL(X,X^*)\neq\{0\}$ by Proposition \ref{p:no}, contradicting the assumption on $\tL$.
Hence there is a solution with $\dim(V)<\eta$.

We now argue by induction that, 
with  $\delta$ fixed, 
for each $\eta\le \de$
and each $\de\times \eta$ affine linear pencil $\pL$ such that 
$\pL(X,X^*)$ is full rank for every $X$ in the interior of $\cD_L(\sigma)$, 
we have  $\pL$ is  full rank on the interior of $\cD_L$.

In the case $\eta=1$, there is a solution to the system \eqref{eq:preSDP}
with $0=\dim(V)<\eta=1$. By Lemma \ref{l:yes}, we conclude that $\tL$ is full
rank on the interior of $\cD_L(\sigma)$. Hence the result
holds for $\eta=1.$

Recall that $\ve\le\delta$ and suppose the result holds for each $\eta<\ve$. Let $\tL$ be 
a $\de\times\ve$ affine linear pencil that is full rank 
on the interior of $\cD_L(\sigma)$.  As seen above, there
is a solution $D$ of \eqref{eq:preSDP} with $\eta=\dim(V)<\ve$.
In the case $\eta=0$, just as before, an application
of Lemma \ref{l:yes} completes the proof. Accordingly,
we assume $0<\eta<\ve$.
Let $\tL'$ denote the $\delta\times \eta$ pencil whose coefficients are the restrictions of the
coefficients of $\tL$ to $V$. Let $X$ satisfy $L(X,X^*)\succ0$ and suppose $\tL(X,X^*)(u+u')=0$ for $u\in V^\perp$ and $u'\in V$. 
Thus,
$$(u+u')^* \real(D\tL)(X,X^*) (u+u')=0$$
and hence, by equation \eqref{e:45},
$$u^*\left(P_0 + \sum_k C_k^* L C_k\right)u=0.$$
Thus $u\in V$ and therefore $u=0.$ 
 Consequently $\tL'(X,X^*)u'=\tL(X,X^*)u'=0$. Therefore, 
 for each $X$ in the interior of $\cD_L$,
\begin{equation}\label{e:46}
\ker\tL(X,X^*)\neq\{0\} \iff \ker\tL'(X,X^*)\neq\{0\}.
\end{equation}
In particular, by assumption if $X$ is in the interior of $\cD_L(\sigma)$,
 then $\ker\tL(X,X^*)=\{0\}$. Hence the same is true of $\tL^\prime$.
By the induction hypothesis, $\tL'$ is of full rank on the interior of $\cD_L$. Therefore $\tL$ is of full rank on the interior of $\cD_L$ by \eqref{e:46}.
\end{proof}

\subsubsection{Algorithm}
\label{sss:algo3}
Let $L$ be a $d\times d$ hermitian monic pencil
and let $\tL$ be a $\de\times\ve$ affine linear pencil. Following the proof of Corollary \ref{c:rank} we describe an algorithm for checking whether $\tL$ is of full rank on the interior of $L$.

{\bf Step 1.}
Solve the following feasibility SDP:
\begin{equation}\label{eq:sdp1}
\begin{split}
\trc(\real(D\tL)(0))&=1\\
\real(D\tL)& = P_0 + \sum_k C_k^* L C_k
\quad
\text{ for some } C_k,P_0, \text{ with }P_0\succeq0.
\end{split}
\end{equation}

We note that \eqref{eq:sdp1} is  a SDP.
Indeed, the first equation is simply a linear constraint, and the second equation can be rewritten 
as a semidefinite constraint
using (localized) moment matrices;
see e.g.~\cite{PNA,BKP} for details.

{\bf Step 2.}
If 
\eqref{eq:sdp1} is infeasible, 
then $\tL(X,X^*)$ is not of full rank for some $X$ in the interior of $\cD_L$ by Proposition \ref{p:no}.

{\bf Step 3.}
Otherwise we have a solution with $V:=\ker P_0 \cap \bigcap_k \ker C_k\subsetneq \C^{\ve}$. 

\hspace{5mm}{\it Step 3.1}
If $V=(0)$, then $\tL$ is of full rank on the interior of $\cD_L$ by Lemma \ref{l:yes}.

\hspace{5mm}{\it Step 3.2.}
If $\ve'=\dim V>0$, then let $\tL'$ be the $\delta\times \ve'$ affine linear pencil whose coefficients are the restrictions of coefficients of $\tL$ to $V$. Then $\tL$ is of full rank on the interior of $\cD_L$ if and only if $\tL'$ is of full rank on the interior of $\cD_L$. Now we apply Step 1 to $\tL'$; since $\tL'$ is of smaller size than $\tL$, the procedure will eventually stop.

\section{Examples}\label{s:exa}
We say that a hermitian $f\in\px$ with $f(0)=1$ is a \df{minimal degree defining polynomial for $\cD_f$} if $\deg h\ge\deg f$ for every hermitian $h\in\px$ such that $\cD_f=\cD_h.$ In this section we present examples of hermitian polynomials $f$ such that $\cD_f$ is a free spectrahedron, $f$ is a minimal degree defining polynomial for $\cD_f$, and $f$ is of degree more than two. By Theorem \ref{thm:main} such an $f$ necessarily factors, even if $\cD_f$ corresponds to an \irrL pencil. The construction of such $f$ relies on the following lemma.

\begin{lem}\label{l:exa}
Suppose   $f_1,s\in \px$ are \irrs  and $L$ is a  hermitian monic pencil.
If\looseness=-1
\ben[label={\rm(\arabic*)}]
\item \label{i:exa1}$s(0)=1=f_1(0)$ and $\deg f_1>2$;
\item \label{i:exa2}
$\cZ_{f_1}=\cZ_L$ and thus $\cK_{f_1}=\cD_L$;
\item \label{i:exa3} $s$ is hermitian;
\item \label{i:exa4}
$f_1 s = s f_1^* $;
\item\label{i:exa5}
$s(X,X^*) \succ 0$ for all $(X,X^*) \in \cD_L$,
\een
then $f:=f_1s$ is hermitian and $\cD_f= \cD_L$. Furthermore, a minimal degree defining polynomial for $\cD_f$ has degree at least $1+\deg f_1$.
\end{lem}

\begin{proof}
The polynomial $f$ is hermitian by items \ref{i:exa3} and \ref{i:exa4},
 and $\cD_f= \cD_L$ holds by item \ref{i:exa2} and \ref{i:exa5}. 
Now let $h$ be an arbitrary hermitian polynomial satisfying $\cD_h=\cD_f.$ 
Let $\tL$ denote a minimal hermitian monic pencil such that
$\cD_{\tL} = \cD_L$. 
By Lemma \ref{p:DZ}\ref{i:DZ2} $\cZ_h\supseteq \cZ_{\tL}$. 
Since $\cK_{f_1} = \cD_{\tL}$, $f_1$ is an atom and $\tL$ is minimal,
$\cZ_{f_1}=\cZ_{\tL}$. 
Thus $\cZ_h\supseteq \cZ_{f_1}$. 
Since $f_1$ is an \irr, 
 $h$ has an \irra factor of degree $\deg f_1$ by \cite[Theorem 4.3(3)]{HKV}. Thus the degree of $h$ exceeds two by item \ref{i:exa1}. 
Hence $h$  is not an \irr by Theorem \ref{thm:main}.  It follows that $\deg h\ge 1+\deg f_1$.
\end{proof}

\begin{rem}
In general, Corollary \ref{cor:sacvx} implies that $f\in\px$ with $f(0)\neq0$ has convex $\cK_f$ if and only if it admits a complete factorization $f=s_0f_1s_1\cdots f_\ell s_\ell$, where $\cK_{f_k}$ are convex (such $f_k$ are characterized in Section \ref{s:flip}) and $s_0\cdots s_\ell$ is invertible on $\bigcap_k \cK_{f_k} = \cK_f$.
\end{rem}

For the rest of this section let $g=1$ and $x=x_1$.

\subsection{Example of degree 4}\label{ss:e1} %p14, RandSS-SD-18JanDeg4.nb

Let
$$f_1 =1 + x + x^* - 2xx^*-(x + x^*)xx^*,
       \qquad  s = 1 + \frac12(x+x^*) $$
 and
$$ L = \bem 1 + x + x^* & 0 & x \\ 0 & 1 & x \\ x^* & x^* & 1 \eem .$$
Let us sketch how to verify the assumptions of Lemma \ref{l:exa}. Clearly, $s$ is an \irr and 
items \ref{i:exa1} and \ref{i:exa3} of Lemma \ref{l:exa} hold.  Using standard realization algorithms (e.g. as in \cite{BGM05}) one checks that $L$ appears in a minimal realization of $f_1^{-1}$. Moreover, a direct computation shows that $L$ is \irrL. Hence $f_1$ is an \irr by Proposition \ref{p:FM}\ref{it:indec}, and item \ref{i:exa2} holds by Proposition \ref{p:FM}\ref{FM1}. Next, item \ref{i:exa4} is straightforward to verify. Finally, for every $(X,X^*)\in\cD_L$ we have $I+X+X^*\succeq0$ and consequently $I+\frac12(X+X^*)\succ0$, so item \ref{i:exa5} holds. 

By Lemma \ref{l:exa}, $f=f_1s$ is hermitian with $\cD_f=\cD_L$, and $f$ is a minimal degree defining polynomial for $\cD_f$ since $\deg f=4=\deg f_1+1$. Note that
$$\{(X,X^*)\colon f(X,X^*)\succeq0 \}\neq\cD_L$$
in this case.

\subsection{Example of degree 5 or 6}\label{ss:e2} %p15, RandSS-SD-18JanDeg6.nb

Let
\begin{align*}
f_1 &= 
1 - (x + x^*) - 2 (x + x^*)^2 - 2 x^*  x  + 
(x+x^*)^3 + 2 
(x+x^*)^2x^*x,
\\
s &= 1 - (x+x^*)^2
\end{align*}
and
$$L = \bem
1 - \frac12(x+x^*) & -\sqrt{2}(x+x^*) & \frac12(x+x^*) & x^* \\[1mm]
-\sqrt{2}(x+x^*) & 1 & 0 & 0 \\[1mm]
\frac12(x+x^*) & 0 & 1 - \frac12(x+x^*) & -x^* \\[1mm]
x & 0 & -x & 1
\eem .$$
As in the previous example the only item of Lemma \ref{l:exa} that is not simple to verify
is \ref{i:exa5}.
Observe that the upper $2\times2$ block of $L$ depends only on the hermitian variable $h=x+x^*$. 
The same holds for $s=1-h^2$.
Hence it suffices to see that $s>0$ on $\cD_{L}(1)$, which is true since
$$\det\begin{pmatrix}1-\frac{\rho}{2} & -\sqrt{2}\rho \\ -\sqrt{2}\rho & 1\end{pmatrix}\ge 0
\quad\implies\quad 1-\rho^2>0$$
for $\rho\in\R$. If $f=f_1s$, then $\cD_f$ is a free spectrahedron domain whose minimal degree defining polynomial has degree at least 5. Note that $\deg f=6$, but we do not know whether $f$ is a minimal degree defining polynomial.

Of course, by taking a Schur complement of $L$ 
we obtain a quadratic $2\times2$ noncommutative polynomial $q$ with $\cD_q=\cD_L$:
\[
q= 
\begin{pmatrix}
1-\frac{x}{2}-\frac{x^*}{2}-2 x^2-2 xx^*-3 x^*x-2
   (x^*)^2 &
   \frac{x}{2}+\frac{x^*}{2}+x^*x \\
 \frac{x}{2}+\frac{x^*}{2}+x^*x &
   1-\frac{x}{2}-\frac{x^*}{2}-x^*x \\
\end{pmatrix}.
\]

\subsection{High degree \protect{\irrs} with convex \protect{$\cK_f$}}\label{ss:e4} 

In the previous two subsections we obtained \irrs $f_1$ of degree $3,4$ 
with convex $\cK_{f_1}$ in agreement with the degree at most four
 conclusion of the main result of \cite{DHM07}.
Nevertheless, it is easy to construct examples of such polynomials $f$ of high degree.

For example, let
\begin{multline*}
f=1 + 4 (x +  x^*) + 2 (x  ^2 +(x^*)^2)- x   x^*  - 7 x   x^* ( x +   x^*)
 - 4 x^*   x  ( x +  x^* )\\ - x   x^*  ( x^2 +  (x^*)^2 )  
 + 2 x   x^*   (x   x^* +
              x^*   x   ) (x+x^*).
\end{multline*}
That  $\cK_f=\cD_L$, where
\[
L=
\begin{pmatrix}
1-x-x^* & x & -x-x^* & x & -x & x+x^* \\
 x^* & 1 & 0 & 0 & 0 & 0 \\
 -x-x^* & 0 & 1+x+x^* & -x & x & -x-x^* \\
 x^* & 0 & -x^* & 1 & 0 & 0 \\
 -x^* & 0 & x^* & 0 & 1 & 0 \\
 x+x^* & 0 & -x-x^* & 0 & 0 & 1+2 x+2 x^*
 \end{pmatrix},
\]
can be checked using realization theory.

\subsection{Counterexample to a one-term Positivstellensatz}\label{ss:e3} %p17

One might hope that for polynomials whose semialgebraic sets are spectrahedra, there exists a one-term Positivstellensatz (cf. \cite[Theorem 1.1]{HKM12}), meaning: if $\cD_f=\cD_L$ for a hermitian polynomial $f$ with $f(0)>0$ and a $d\times d$ hermitian monic pencil $L$, then there exists 
$W\in\matt{d\times d}{\px}$ such that
\begin{equation}\label{e:posss}
I_d\otimes f=f\oplus\cdots\oplus f=W^*LW.
\end{equation}
We note that such a conclusion holds 
for $f$ that are real parts of a noncommutative analytic function under natural irreducibility and minimality assumptions on $L$.  For a proof we refer the gentle reader to \cite{AHKM18}, where this fact  is exploited to characterize
bianalytic maps between free spectrahedra.
However, with Example \ref{ss:e1} we shall demonstrate that \eqref{e:posss} does not hold in general. 

Let us assume the notation of Example \ref{ss:e1} and suppose there exists $W\in\px^{3\times 3}$ such that
\begin{equation}\label{e:f3}
\bem f & 0 & 0 \\ 0 & f & 0 \\ 0 & 0 & f \eem = W^*LW.
\end{equation}
Let $\Omega^{(n)}$ and $\Upsilon^{(n)}$ be $g$-tuples of $n\times n$ generic matrices and consider evaluations of $f,W,L$ at $(\Omega^{(n)},\Upsilon^{(n)})$.\index{$\Upsilon^{(n)}$} Taking determinants of both sides of \eqref{e:f3} gives
$$ \left(\det f(\Omega^{(n)},\Upsilon^{(n)})\right)^3 
= \det W^*(\Omega^{(n)},\Upsilon^{(n)})\det L(\Omega^{(n)},\Upsilon^{(n)})\det W(\Omega^{(n)},\Upsilon^{(n)}).$$
Since $\det L(\Omega^{(n)},\Upsilon^{(n)})=\det f_1(\Omega^{(n)},\Upsilon^{(n)})$, 
\begin{equation}\label{e:2v3}
\left(\det f_1(\Omega^{(n)},\Upsilon^{(n)})\right)^2 \left(\det s(\Omega^{(n)},\Upsilon^{(n)})\right)^3 
= \det W^*(\Omega^{(n)},\Upsilon^{(n)})\det W(\Omega^{(n)},\Upsilon^{(n)}).
\end{equation}
Recall that $s = 1 + \frac12(x+x^*)$, so $p=\det s(\Omega^{(n)},\Upsilon^{(n)})$ is an irreducible polynomial for all $n\in\N$. Therefore it divides $\det W^*(\Omega^{(n)},\Upsilon^{(n)})$ or $\det W(\Omega^{(n)},\Upsilon^{(n)})$ by \eqref{e:2v3}. But $s$ is a hermitian polynomial, so $p$ divides $\det W^*(\Omega^{(n)},\Upsilon^{(n)})$ and $\det W(\Omega^{(n)},\Upsilon^{(n)})$. Therefore the left-hand side of \eqref{e:2v3} is divisible by $p^3$ but not by $p^4$, while the highest power of $p$ dividing the right-hand side of \eqref{e:2v3} is even, a contradiction.

\subsection{High degree matrix \protect{\irrs} defining free spectrahedra}\label{s:exb}

It is fairly easy to produce examples of \irrL hermitian matrix polynomials $F$ of arbitrary high degree such that $\cD_F$ is a free spectrahedron. For example, let $p\in\matt{\de}{\px}\setminus \mat{\de}$ be arbitrary and let
$$F=\begin{pmatrix} I & 0 & x \\ 0 & I & p \\ x^* & p^* & I+p^*p \end{pmatrix}.$$
Then $\deg F= 2\deg p$ and $\det F(\Omega^{(n)},\Upsilon^{(n)})=\det (I-\Upsilon^{(n)}\Omega^{(n)})$ is irreducible for all $n\in\N$, so $F$ is an \irr. Further, $\cD_F=\cD_{1-x^*x}$ is a free spectrahedron.

\section{Classifying hermitian flip-poly pencils}
\label{s:flip}

A byproduct of investigations in earlier sections is a description of hermitian monic flip-poly pencils, which helped us construct Examples \ref{ss:e1}, \ref{ss:e2} and \ref{ss:e4}. Since it is of independent interest, we present it here in more detail.

A $d\times d$ monic pencil $L=I-A\mydot x$ is called 
\df{flip-poly} \cite[Section 5.3]{HKV} if 
\[ 
A_j = N_j + v_j u^*
\]
where the $N_j$ are jointly nilpotent  $d \times d$ matrices  and
$u,v_j\in\C^d$. Such pencils are important for distinguishing free loci of  polynomials among all free loci.

\begin{prop}[{\cite[Corollary 5.5]{HKV}}]\label{p:fp}
The set of free loci of  polynomials coincides with the set
of free loci of flip-poly pencils.

\end{prop}

In this section we further examine the structure of \textit{hermitian} flip-poly pencils. If $L=I-A\mydot x-A^*\mydot x^*$ is a $d\times d$ flip-poly pencil, then by the definition above there exist jointly nilpotent matrices $N_1,\dots,N_g,\tN_1,\dots,\tN_g$ and vectors $u,v_1,\dots,v_g,\tv_1,\dots,\tv_g$ such that
$$A_j=N_j+v_ju^*,\qquad A_j^*=\tN_j+\tv_ju^*.$$

The following folklore statement is a consequence of  Engel's theorem \cite[Corollary 3.3]{Hum} and the Gram-Schmidt process.

\bel
\label{q:nolan}
    Given jointly nilpotent matrices, there is an orthonormal  basis 
    in which they are simultaneously strictly upper triangular.
\eel

After a unitary change of basis (which preserves the hermitian property of $L$) we can therefore assume that $N_j,\tN_j$ are strictly upper triangular matrices. For every $j$, 
$$0=A_j-(A_j^*)^*= N_j + v_j u^*- \tN_j^* - u \tv_j^*,$$
or equivalently,
\beq
\label{eq:skewu}
N_j - \tN_j^*= u \tv_j^*-v_j u^*.
\eeq
On the left-hand side of \eqref{eq:skewu} there is a matrix with diagonal
identically 0. Looking at the right-hand side of \eqref{eq:skewu} we then obtain 
\beq
\label{eq:diag}
u^k\overline{\tv_j^k}= \overline{u^k} v_j^k,
\eeq
for every $1\le j\le g$ and $1\le k\le d$, where $v^k$ denotes the $k$\textsuperscript{th} component of $v$.

Conversely, let $u,v_1,\dots,v_g\in\C^d$ be arbitrary. Next we choose $\tv_1,\dots,\tv_g$ that satisfy equations \eqref{eq:diag}. Observe that this can always be done: if $u^k\neq0$, then $\tv_j^k$ is determined by $u^k$ and $v_j^k$; and if $u^k=0$, then we can choose an arbitrary value for $\tv_j^k$. Then the matrices $u \tv_j^*-v_j u^*$ have diagonals identically 0. Hence by declaring $N_j$ to be the strictly upper triangular part of $u \tv_j^*-v_j u^*$, we obtain matrices $A_j=N_j+v_ju^*$ such that $L=I-A\mydot x-A^*\mydot x^*$ is flip-poly.

Thus we derived the following result.

\begin{prop}\label{p:hfp}
Let $L=I-A\mydot x-A^*\mydot x^*$. Then $L$ is flip-poly if and only if there exist vectors $u,v_1,\dots,v_g$ such that, after a unitary change of coordinates, $A_j=N_j+v_ju^*$, with $N_j$ being the strictly upper triangular part of the matrix $u \tv_j^*-v_j u^*$, where $\tv_j$ is a vector satisfying
$$\tv_j^k=
\frac{u^k \overline{v_j^k}}{\overline{u^k}}\qquad \text{for } u^k\neq0.$$
\end{prop}

\begin{rem}\label{r:real}
Note that vectors $\tv_j$ in Proposition \ref{p:hfp} are uniquely determined if all the entries of $u$ are nonzero. Furthermore, if one is only interested in symmetric pencils, i.e., hermitian pencils with real entries, then the form of $L$ can be further simplified when $u\in(\R\setminus\{0\})^d$. Namely, in this case one has $\tv_j=v_j$ for all $j$. Moreover, since matrices $u v_j^*-v_j u^*$ are skew-symmetric, it follows by \eqref{eq:skewu} that $A_j=A_j^*$ for all $j$. Thus in this situation one has $L=I-A\mydot (x+x^*)$ for symmetric $A_j$; in particular, $\cD_L$ is unbounded. 
Of course, in general not every symmetric flip-poly pencil is of this form, see Examples \ref{ss:e1} and \ref{ss:e2}.
\qed\end{rem}

\section{Hereditary polynomials}\label{s:hered}

We say that a noncommutative polynomial $f$ is \df{hereditary} if it is a linear combination of words $uv$ with $u\in\mxs$ and $v\in\mx$. Furthermore, $f$ is \df{truly hereditary} if it is not analytic or anti-analytic, i.e., $f\notin \pxa\cup\, \pxs$.
Hereditary polynomials arise naturally in
free function theory \cite{Gre};
they are a tame analog of free real 
analytic functions.
For example,  the composite of
an analytic polynomial (with no $x^*$)
with an hermitian pencil, a heavily
studied class of objects in
the geometry of free convex sets (cf.~\cite{AHKM18}), is hereditary.
Similarly, the hereditary functional calculus
\cite{Agl} is a powerful tool in operator theory
and complex analysis.

In this section we prove the following.

\begin{theorem}
\label{thm:hered}
Let $f$ be a hereditary polynomial and $f(0)=1$. Then $f$ admits a unique factorization
\begin{equation}
\label{eq:phq}
f= phq,\qquad p(0)=h(0)=q(0)=1,
\end{equation}
with $p$ anti-analytic, $q$ analytic, and $h$ a truly hereditary \irr or constant. If $f$ is moreover hermitian, then $q=p^*$ and $h=h^*$.
\end{theorem}

The normalization $p(0)=h(0)=q(0)=1$ is only required to avoid ``uniqueness up to scaling''. Before giving a proof of Theorem \ref{thm:hered} we record the following corollary.

\begin{cor}
\label{cor:heredDefpoly}
Any hereditary minimal degree defining polynomial for a free spectrahedron  is an \irr,
and hence has degree at most 2.
\end{cor}

\begin{proof}
Let $f$ be hereditary and minimal degree defining polynomial for $\cD_f$, and let $\cD_f=\cD_L$ for a minimal hermitian monic pencil $L$. Therefore $\partial\cD_L\subseteq \cZ_f$ and hence $\cZ_L\subseteq \cZ_f$ by Proposition \ref{p:pd-min}. Furthermore, after a unitary change of basis we can assume that $L=L^1\oplus\cdots\oplus L^\ell$, where the $L^i$ are pairwise non-similar \irrL hermitian pencils. Observe that for each $i$ and large enough $n$, the polynomial $\det L^i(\Omega^{(n)},\Upsilon^{(n)})$ is irreducible by Proposition \ref{p:irrpoly} and cannot be independent of $\Omega^{(n)}$ or $\Upsilon^{(n)}$ by \cite[Proposition 3.3]{KV}, where $\Omega^{(n)}$ and $\Upsilon^{(n)}$ are $g$-tuples of $n\times n$ generic matrices corresponding to evaluating variables $x$ and $x^*$, respectively, in an involution-free way.

By Theorem \ref{thm:hered},  $f = a^* h a$, where $a$ is analytic (contains only variables $x$), and $h$ is a hermitian hereditary \irr. Since
$$\cZ_{L^1}\cup\cdots\cup\cZ_{L^\ell}=\cZ_L\subseteq \cZ_f=\cZ_a\cup\cZ_h\cup\cZ_{a^*},$$
for each $i$ and large enough $n$, the irreducible polynomial $\det L^i(\Omega^{(n)},\Upsilon^{(n)})$ divides one of the polynomials $\det a^*(\Upsilon^{(n)})$, $\det h(\Omega^{(n)},\Upsilon^{(n)})$ or $\det a(\Omega^{(n)})$. By the previous paragraph, it cannot divide the first one or the last one, so $\cZ_{L^i}\subseteq \cZ_h$.
Since $h$ is an \irr it follows that $\cZ_{L^i} = \cZ_h$. Because the $L^i$ are pairwise non-similar \irrL pencils, we necessarily have $\ell=1$, so $L$ is \irrL. Therefore $\cD_h = \cD_L$ by Proposition \ref{p:DZ}\ref{i:DZ3}. Thus,  $h$ is concave of degree at most two by Theorem \ref{thm:main}. Finally, since $f$ is of minimal degree, $a = 1$ and $f = h$.
\end{proof}

\begin{cor}
\label{cor:hered2}
If $q\in\pxa$ and $\cD_{q+q^*}$ is a free spectrahedron, then $\deg(q)\leq1$.
\end{cor}

\begin{proof}
Observe that $q+q^*$ is an atom in $\px$ for every non-constant $q\in\pxa$. Therefore $q+q^*$ is of degree at most 2 and concave by Theorem \ref{thm:main}, so
$$q+q^*=\alpha+\ell-\sum_k \ell_k^*\ell_k$$
for some $\alpha>0$ and linear polynomials $\ell,\ell_k\in\px$. If some $\ell_k$ is nonzero, then $\ell-\sum_k \ell_k^*\ell_k$ has a term of the form $\alpha x_jx_j^*$ or $\alpha x_j^*x_j$ with $\alpha<0$. On the other hand, there are no mixed terms in $q+q^*$, so we conclude that $\ell_k=0$ for all $k$. Therefore $q$ is affine linear.
\end{proof}

\def\vk{{\widehat k}}
\def\vell{{\widehat \ell}}
\def\vh{{\widehat h}}  
\def\vp{{\widehat p}}
\def\vq{{\widehat q}}

\ssec{Proof of existence of the factorization \protect{(\ref{eq:phq})}}

\begin{lem}\label{l:fac}
 Suppose $f$ is hereditary and $f=pq$. If $p\notin \pxs$, then $q\in \pxa$.
 If $f=a^*hb$ and $a,b\in \pxa$, then $h$ is hereditary.
\end{lem}

\begin{proof}
 To prove the first statement, suppose $p\notin\pxs$ and $q\notin\pxa$.
 Write, $p=\sum p_{\alpha} \alpha$ and $q=\sum q_\beta \beta$. 
 There exists a word $\alpha'$ and a $j$ such that $\alpha'$ contains $x_j$
 and $p_{\alpha'} \ne 0$; 
 and there is a word $\beta'$ and a $k$ such that $\beta'$ contains $x_k^*$ 
 and $q_{\beta'} \ne 0$.  Without loss of  we may assume that 
 the (total) degrees of $\alpha'$ and $\beta'$ are maximal with these properties.
 Now,
\[
  f= \sum_{\gamma} \big(\sum_{\alpha\beta=\gamma} p_\alpha q_\beta \big) \gamma.
\]
Let  $\Gamma = \alpha'\beta'$ and note that this word is not
hereditary.  Thus,
\[
 \sum_{\alpha\beta =\Gamma} p_\alpha q_\beta = 0.
\]
 It follows that there exists words $\sigma$ and $\tau$ such
 that $(\sigma,\tau)\ne (\alpha',\beta')$, $p_\sigma\ne 0,$ $q_\tau \ne 0$ and $\Gamma =\sigma\tau = \alpha'\beta'$.
 It follows that either $\alpha'$ properly divides $\sigma$ on the left,
 in which case $\sigma$ contains $x_j$ and $|\sigma|>|\alpha'|$,
 contradicting the choice of $\alpha'$; or $\beta_m$ properly divides
 $\tau$ on the right, in which case $\tau$ contains $x_k^*$ and $|\tau|>|\beta'|$,
 contradicting the choice of $\beta'$.

 The second statement can be proved in a similar fashion.  Sketching the argument,  write
\[
 h = \sum h_\beta \beta
\]
 and, arguing by contradiction, suppose there is a $\beta'$ such that $h_{\beta'}\ne 0$ 
 has an $x$ to the left of an $x^*$. Let $\alpha'$ and $\gamma'$ denote 
 maximum degree terms in $a^*$ and $b$. It follows that $\alpha' \beta' \gamma'$
 must appear in $a^* h b$ (and has largest degree amongst words in $a^* h b$
 containing an $x$ to the left of an $x^*$) and thus $f$ is not hereditary.
\end{proof}

\begin{proof}[Proof of existence in Theorem \ref{thm:hered}] 
  The hereditary polynomial $p$ factors as
\[ 
f= q_0 q_1 q_2 \dots q_s q_{s+1},\qquad q_k(0)=1,
\] 
where $q_0=1=q_{s+1}$  and, for each $1\le j\le s$,
the factor $q_j$ is an atom. Suppose, without loss of generality, that
$f\notin\pxa$.  There is an $1\le r\le s$ such that 
$q_{r+1}\cdots q_{s+1} \in \pxa$, but  $q_r q_{r+1}\cdots q_{s+1}\notin\pxa$.
By Lemma \ref{l:fac}, $q_0q_1 \cdots q_{r-1} \in \pxs$ as
$f= (q_0 q_1 \cdots q_{r-1})\,(q_r \cdots q_{s+1})$ is hereditary. 
Thus $f=a^* h b$, where $a = (q_0 q_1\cdots q_{r-1})^*, q_{r+1}\cdots q_{s+1}\in \pxa$
and $h=q_r$. By the other half of Lemma \ref{l:fac}, $q_r$ is hereditary and the proof 
is complete.
\end{proof}

\ssec{Proof of uniqueness of the factorization \protect{(\ref{eq:phq})}}
\label{sec:unique}

Proving uniqueness requires background from Cohn \cite{Coh06} which we now introduce.

\def\hq{\widehat{q}}

Let $q_1, q_2, \hq_1,\hq_2\in\pxa$ and suppose

\begin{equation}\label{e:rel}
q_1 q_2= \hq_1\hq_2.
\end{equation}
If
\[
q_1\pxa+ \hq_1 \pxa=\pxa, \qquad \pxa  q_2+\pxa  \hq_2=\pxa,
\]
then \eqref{e:rel} is called a \df{comaximal relation} \cite[Section 0.5]{Coh06}. 
If, moreover, $ q_1, q_2, \hq_1, \hq_2$ are \irrs and 
\[ 
q_1\pxa\cap \; \hq_1\pxa \ \mbox{is a principal right ideal in} \ \pxa,
\] 
then \eqref{e:rel} is called a \df{comaximal transposition} \cite[Section 3.2]{Coh06}.

Next, $q_1,\hq_2$ are \df{stably associated} \cite[Section 0.5]{Coh06} if
\[ 
 I_d\otimes \hq_2=P(I_d\otimes q_1)Q,
\] 
for some $d\in\N$ and $P,Q\in\operatorname{GL}_{d+1}(\pxa)$.

\begin{prop}[{\cite[Proposition 0.5.6]{Coh06}}]
\label{prop:cohn1} \mbox{}
$ q_1$ and $ \hq_2$ are stably associated 
if and only if 
they appear in a comaximal relation 
\eqref{e:rel} for some $ q_2, \hq_1$.
\end{prop}

Finally, a factorization $f=f_1\cdots f_\ell$ in $\pxa$ is {\bf complete}\index{complete factorization} \cite[Section 3.2]{Coh06} if the $f_k$ are \irrs. Two complete factorizations of $f$ are identified if their factors only differ up to scalars. Note that a noncommutative polynomial can admit distinct complete factorizations, e.g.
$$(1+x_1x_2)x_1=x_1(1+x_2x_1).$$
However, this relation is a comaximal transposition. In fact, the following holds.

\begin{prop}[{\cite[Proposition 3.2.9]{Coh06}}]
\label{prop:cohn2}
Given two complete factorizations of a polynomial, one can pass between them by a finite sequence of comaximal transpositions on adjacent pairs of \irra factors (in particular, they have the same length).
\end{prop}

Let us illustrate what is meant by a {\it finite sequence of comaximal transpositions.}  To say that $q_1q_2q_3q_4$ is a complete factorization that can be transformed to a different factorization by applying comaximal transpositions on positions $(2,3)$, $(3,4)$ and $(1,2)$ (in this order)  means there exists $\hq_2,\hq_3,\widehat{\hq}_3,\widehat{\hq}_2$ such that
$$q_1q_2q_3q_4=
q_1\hq_2\hq_3q_4=
q_1\hq_2\widehat{\hq}_3\hq_4=
\hq_1\widehat{\hq}_2\widehat{\hq}_3\hq_4,$$
where
$$q_2q_3=\hq_2\hq_3,\qquad
\hq_3q_4=\widehat{\hq}_3\hq_4,\qquad
q_1\hq_2=\hq_1\widehat{\hq}_2$$
are comaximal transpositions.

\bel
\label{l:notrans}
Suppose $ \ell h =f_1 f_2$ is a comaximal relation
where $\ell\in\pxs$, $h$ is hereditary,  
$f_1, f_2 \in \px$ 
and all are normalized to equal $1$ at the origin. Then $f_1, f_2,h\in\pxs$. 

Analogously, if $h r  =f_1 f_2$ is a comaximal relation with $r\in\pxa$ and $h$ hereditary, then $f_1, f_2,h\in\pxa$.
\eel

\bep
By Proposition \ref{prop:cohn1}, $\ell$ and $f_2$ are stably associated. 
Then by the definition of stable associativity there exists $\alpha\in\C\setminus\{0\}$ such that
$$\det \ell(\Upsilon^{(n)})=\alpha^n \det f_2(\Omega^{(n)},\Upsilon^{(n)})$$
for all $n\in\N$, where $\Omega^{(n)}$ and $\Upsilon^{(n)}$ are tuples of $n\times n$ generic matrices.	By \cite[Proposition 5.11]{HKV},  $f_2\in \pxs$. But $f_1f_2= \ell h$ is hereditary, so $f_1 \in \pxs$ and consequently $h \in \pxs$.
	\eep

\begin{proof}[Proof of uniqueness in Theorem \ref{thm:hered}]
Suppose $f= p h q = \vp\vh\vq$ are two  factorizations as in Theorem \ref{thm:hered}.
Let
$$p=p_1\cdots p_k,\qquad \vp=\vp_1\cdots\vp_\vk,\qquad
q=q_1\cdots q_\ell,\qquad \vq=\vq_1\cdots\vq_\vell$$
be complete factorizations (with factors equal to 1 at the origin). Then
\begin{equation}\label{e:long}
p_1\cdots p_khq_1\cdots q_\ell \ 
=  \ \vp_1\cdots\vp_\vk \vh \vq_1\cdots\vq_\vell.
\end{equation}
and by Proposition \ref{prop:cohn2} we can pass from the left-hand side to the right-hand side of \eqref{e:long} by a series of comaximal transpositions. The heart  of the proof is that there cannot be any transposing
around the ``middle" factor $h$ unless it is trivial. 
Since $f$ and all the factors $p,q,h$ are normalized to equal $1$ at $0$, we can apply
 Lemma \ref{l:notrans} to conclude the proof:
 for if we can transpose $p_k h$, then $h\in\pxs$ and so $h=1$ since $h$ is truly hereditary. Likewise for $h q_1$. When $h$ is not trivial, comaximal transpositions can therefore only occur among the first $k-1$ factors and last $\ell-1$ factors of the left-hand side in \eqref{e:long}. However, these comaximal transpositions preserve $p_1\cdots p_k$ and $q_1\cdots q_\ell$. Thus we conclude that $p_1\cdots p_k \ =   \vp_1\cdots\vp_\vk$ and $q_1\cdots q_\ell \ =   \vq_1\cdots\vq_\vell$. Therefore $p=\vp$ and $q=\vq$, and consequently $h=\vh$.

The last part of Theorem \ref{thm:hered} is a direct consequence of the uniqueness.
\end{proof}

\appendix

\section{Modification of the theory: rational functions}
\label{sec:modify}

For the reader familiar with nc rational functions as found in \cite{Coh06,KVV09}, we point out that
Theorem \ref{thm:main2} extends to matrix noncommutative rational functions in a straightforward way. Assume $\rr\in\rxx^{\de\times\de}$ is regular 
 at the origin (that is, $0$ is in the domain of $\rr$)
  and $\rr(0)=I$. Then we define $\cK_\rr=\bigcup_n\cK_\rr(n)$, where $\cK_\rr(n)$ is the closure of the connected component of\looseness=-1
\[\left\{(X,X^*)\in\mat{n}^{2g} \colon \rr \text{ is regular at } (X,X^*) \text{ and }\det \rr(X,X^*)\neq 0\right\}\]
containing the origin.

Now let $I+c^*L^{-1}\bb$ be a minimal FM realization for $\rr\oplus\rr^{-1}\in\rx^{2\de\times2\de}$. Using Remark \ref{r:facts}\ref{it:rf3} we observe that $\cZ_L$ is precisely the set of all $(X,X^*)$ for which either $\rr$ is not defined at $(X,X^*)$ or $\rr$ is regular at $(X,X^*)$ and $\det\rr(X,X^*)=0$. By comparing this observation with the definition of $\cK_\rr$, we see that 
\beq\label{eq:rL}
\cK_\rr=\cK_L.
\eeq 
Now we apply the proof of Theorem \ref{thm:main2} to $L$.

 Likewise, 
from \eqref{eq:rL} we deduce that 
 Corollary \ref{cor:algo12}
 holds for rational functions $\rr$.
This leads to improvements and strengthening of
recent positivity results for noncommutative
rational functions \cite{KPV17,Pas18}.
For instance, a rational function $\rr$
is positive definite on the interior
of $\cD_{L}$ if and only if
$\rr(0)\succ0$ and 
$\tL$ is invertible on 
$\interior\cD_{L}$, where
$\tL$ is the minimal pencil in an FM realization of $\rr\oplus\rr^{-1}$.
The latter condition can be efficiently checked
by the algorithm of Subsection \ref{ssec:algo2}.

In \cite{Pas18}, Pascoe
gives a Positivstellensatz certifying 
when a noncommutative rational function
$\rr$ that is defined on $\cD_L$, is
positive semidefinite on $\cD_L$.
For bounded $\cD_L$ our algorithms provide
means of verifying whether $\rr$ is
defined on $\cD_L$. Let $\tL$ be
the minimal pencil in an FM realization of $\rr$.
Then $\tL$ is invertible on
$\cD_L$ if and only if
there is $\ve>0$ such that
$\tL \tL^*-\ve$ is invertible on $\interior \cD_L$, and this is something that can be checked
with a sequence of SDPs (cf.~Subsection \ref{ssec:algo2}).

We conclude with a variant of Theorem \ref{thm:main} for rational functions.
\df{McMillan degree} (\cite{KVV09}) of a rational function is the size of the linear pencil in
its minimal FM realization.
Lemma \ref{l:McMin} below asserts that, given $L$ a minimal hermitian monic pencil $L$,
there exists a hermitian $\ss\in\rxx$ such that $\cK_\ss = \cD_L$. We say that a hermitian $\rr\in \rxx$ is
\df{minimal (McMillan) degree  defining for $\cD_L$} if $\cK_{\rr}=\cD_L$ 
and the McMillan degree of $\rr$ is smallest amongst all hermitian $\ss$ such that 
$\cK_{\ss}=\cD_L$. 

\begin{prop}\label{prop:main}
Let $\rr=\rr^*\in\rxx$ be regular at the origin and $\rr(0)=1$. Suppose
that $\cK_\rr$ is a free spectrahedron
$\cD_L$ for an \irrL hermitian monic pencil $L$. 
If $\rr$ is minimal McMillan degree defining for $\cD_L$, 
then either $\rr$ or $\rr^{-1}$ is
concave or convex with the pencil   in its minimal FM realization being equal to $L$.
\end{prop}

\begin{lemma}
\label{l:McMin}
 Suppose $L$ is an \irrL hermitian monic pencil of size $d$ and $0\ne \hc\in \C^d$
 is of norm $<1$.
 Setting $\widehat\bb=Ac\mydot x+A^*c\mydot x^*$ and
$\widehat\rr=1+\widehat c^* L^{-1} \widehat\bb,$
\[
\cK_{\widehat{\rr}} = \cD_L,
\]
 $\widehat{\rr}^{-1}$ is defined on $\interior \cD_L$ and $\widehat{\rr}^*=\widehat{\rr}$.
\end{lemma}

\begin{proof}
Since the converse of Lemma \ref{l:exp} evidently holds, $\rr^*=\rr$.
By Remark \ref{r:facts}\ref{it:rf5} we have $\widehat\rr^{-1}=1-\widehat c^* L_\times^{-1} \widehat\bb$, where $L_\times=L+\widehat \bb\widehat c^*$. Since $L$ is \irrL and $\widehat c\ne 0$ and $\widehat{\bb} \ne 0$, 
the realization $\widehat\rr=1+\widehat c^* L^{-1} \widehat\bb$ is observable and controllable, and thus minimal by Remark \ref{r:facts}\ref{it:rf1}. Consequently $\widehat\rr^{-1}=1-\widehat c^* L_\times^{-1} \widehat\bb$ is also minimal.
The pencil $L_\times$ is invertible on $\interior \cD_L$ because
\[
(I-\widehat c\widehat c^*)(L+\widehat \bb\widehat c^*)
=(I-\widehat c\widehat c^*)\widehat c\widehat c^*+(I-\widehat c\widehat c^*)L(I-\widehat c\widehat c^*).
\]
By the definition of $\cK_{\widehat\rr}$ we have
$$\cK_{\widehat\rr}=\cK_{L\oplus L_\times},$$
so invertibility of $L_\times$ on $\interior\cD_L$ implies
$$\cK_{\widehat\rr}=\cD_L.$$ 
Furthermore, the domain of $\widehat{\rr}^{-1}$ is the complement of $\cZ_{L_\times}$ by 
Remark \ref{r:facts}\ref{it:rf3}, so $\widehat{\rr}^{-1}$ is defined on $\interior \cD_L$.
\end{proof}

\begin{proof}[Proof of Proposition~\ref{prop:main}]
Let $L=I-A\mydot x-A^* \mydot x^*$ be of size $d$.
 Let $\rr=1+c^* \tL^{-1} \bb$ be a minimal realization.
Hence $\rr^{-1} = 1 - c^* \tLx^{-1} \bb$, where $\tLx$ is the pencil
appearing in Remark \ref{r:facts}\ref{it:rf5}, is a minimal realization
for $\rr^{-1}$.  Since $\cK_\rr=\cD_L$, the topological boundary of $\cD_L$ is contained in
\[
\{(X,Y)\colon \rr \text{ is undefined at } (X,Y)\}
\cup
\{(X,Y)\colon \rr^{-1} \text{ is undefined at } (X,Y)\}
 = \cZ_{\tL} \cup \cZ_{\tLx}.
\] 
Since $L$ is an \irrL hermitian monic pencil, it is minimal. Thus,
by Proposition \ref{p:pd-min}, $\cZ_L\subseteq \cZ_{\tL}\cup \cZ_{\tLx}$.
Since $L$ is \irrL, either $\cZ_L\subseteq \cZ_{\tL}$ or
 $\cZ_L\subseteq \cZ_{\tLx}$. Without loss of
generality suppose $\cZ_L\subseteq\cZ_{\tL}$ (otherwise replace $\rr$ by $\rr^{-1}$). Since $L$ is \irrL,
up to similarity (change of basis), 
$\tL$ has the form \eqref{e:triang}, where one of the blocks equals $L$.
On the other hand, by Lemma \ref{l:McMin}, the size of $\tL$ is no larger
than the size of $L$.  Hence $\tL$ is similar to $L$ and we may assume,
by modifying $c,\bb$ and $A$ appropriately, that $\tL=L$. Therefore,
as $L$ is an \irrL hermitian  monic  pencil, 
$$\rr
 =1+\lambda c^* L^{-1}(Ac\mydot x+A^*c\mydot x^*)
 =1+\lambda(c^*L^{-1}c-c^*c)
 $$
 for some $\lambda\in\R\setminus\{0\}$ by Lemma \ref{l:exp}. 
 Since $L$ is monic and hermitian, $\rr$ is concave or convex (depending on the sign of $\lambda$).
\end{proof}

\end{document}